\documentclass[11pt, oneside, reqno]{amsart}

\theoremstyle{plain}
\newtheorem{theorem}{Theorem}[section]
\newtheorem{proposition}[theorem]{Proposition}
\newtheorem*{theorem_i}{Theorem}
\newtheorem{claim}{Claim}
\newtheorem{corollary}[theorem]{Corollary}

\theoremstyle{definition}
\newtheorem{definition}[theorem]{Definition}
\newtheorem{example}[theorem]{Example}

\theoremstyle{remark}
\newtheorem{remark}[theorem]{Remark}
\newtheorem{notation}[theorem]{Notation}
\newtheorem*{acknowledgment*}{Acknowledgment}  

\usepackage{mathrsfs}
\usepackage{bm}
\usepackage{amsmath}
\usepackage{a4wide}
\usepackage{amssymb}
\usepackage{amstext}

\newcommand{\weight}{\mathop{\mathrm{wt}}\nolimits}
\newcommand{\Glow}{G^{\mathrm{low}}}
\newcommand{\Gup}{G^{\mathrm{up}}}
\newcommand{\dprod}{\mathop{\overrightarrow{\prod}}\limits}
\newcommand{\Image}{\mathop{\mathrm{Im}}\nolimits}
\newcommand{\Hom}{\mathop{\mathrm{Hom}}\nolimits}
\newcommand{\Aut}{\mathop{\mathrm{Aut}}\nolimits}
\newcommand{\Ker}{\mathop{\mathrm{Ker}}\nolimits}
\newcommand{\spn}{\mathop{\mathrm{span}}\nolimits}
\newcommand{\fraction}{\mathop{\mathrm{Frac}}\nolimits}

\makeatletter
    
    \@addtoreset{equation}{section}
  \makeatother
\begin{document}
%%   Here starts the topmatter.
%% Text must not happen before \maketitle command!

\title[The Chamber Ansatz for quantum unipotent cells]
{The Chamber Ansatz for quantum unipotent cells}
\author{Hironori OYA}

\address[Hironori OYA]{Universit\'e Paris Diderot, Sorbonne Universit\'e, CNRS, Institut de Math\'ematiques de Jussieu - Paris Rive Gauche, IMJ-PRG, F-75013, Paris, FRANCE}

\email{hironori.oya@imj-prg.fr}
\thanks{The work of the author was supported by Grant-in-Aid for JSPS Fellows (No.~15J09231) and the Program for Leading Graduate Schools, MEXT, Japan. It was also supported by the European Research Council under the European Union's Framework Programme H2020 with ERC Grant Agreement number 647353 Qaffine, during the revision of this paper.}
\date{}

\maketitle

\begin{abstract}
In this paper, we prove quantum analogues of the Chamber Ansatz formulae for unipotent cells. These formulae imply that the quantum twist automorphisms, constructed by Kimura and the author, are generalizations of Berenstein-Rupel's quantum twist automorphisms for unipotent cells associated with the squares of acyclic Coxeter elements. This conclusion implies that the known compatibility between quantum twist automorphisms and dual canonical bases corresponds to the property conjectured by Berenstein and Rupel. 
\end{abstract}
%%   We insist on abstract in your paper!
\section{Introduction}\label{s:introduction}
\subsection{About this subject}\label{introabst}Totally positive elements of an arbitrary  connected semisimple algebraic group $G$ and some related varieties were introduced by Lusztig \cite{Lus:pos}. They are generalizations of totally positive matrices, which are defined as the square matrices such that all their minors are positive. The present paper concerns totally positive elements in \emph{unipotent cells}. A unipotent cell $N_-^w$ is a certain subvariey of a maximal unipotent subgroup $N_-$ of $G$ associated with an element $w$ of  the corresponding Weyl group. In \cite{Lus:pos}, Lusztig has proved that totally positive elements in a unipotent cell are parametrized by the tuple of positive real numbers via a birational map from an algebraic torus to a unipotent cell.
Berenstein, Fomin and Zelevinsky \cite{BFZ,BZ:Schu} have given effective criteria of total positivity in unipotent cells. Their research leads to their definition of \emph{cluster algebras}. The key step for the proof of their criteria is an explicit description of the inverse of the birational map above, which is called \emph{the Chamber Ansatz formulae}. The Chamber Ansatz formulae are given by \emph{the generalized minors} and \emph{twist automorphisms} on unipotent cells. In this paper, we consider the quantum analogues of all these situations. 
\subsection{Chamber Ansatz formulae}Let $\mathfrak{g}$ be a complex finite dimensional semisimple Lie algebra. (From Section 2, we consider the case that $\mathfrak{g}$ is an arbitrary symmetrizable Kac-Moody Lie algebra.) Let $\mathfrak{g}=\mathfrak{n}_-\oplus\mathfrak{h}\oplus \mathfrak{n}_+$ be a triangular decomposition of $\mathfrak{g}$, $G$ the corresponding simply-connencted connected algebraic group, $N_{\pm}$, $H$ the closed subgroups with Lie algebras $\mathfrak{n}_{\pm}$, $\mathfrak{h}$, respectively, $B_{\pm}:=HN_{\pm}$ the Borel subgroups and $W:=N_{G}\left(H\right)/H$ the Weyl group of $\mathfrak{g}$. For $w\in W$, \emph{a unipotent cell} is defined as the algebraic subvariety $N_-^w:=N_-\cap B_+\dot{w}B_+$ of $N_-$, where $\dot{w}$ is an arbitrary lift of $w$ to $N_G(H)$. Let $\{\alpha_i\mid i\in I\}$ (resp.~$\{h_i\mid i\in I\}$) be the set of the simple roots (resp.~coroots) of $\mathfrak{g}$, and $\{s_i\mid i\in I\}$ the set of simple reflections of $W$. Let $f_{i}$ be a root vector corresponding to the root $-\alpha_i$, and $\mathbb{C}\to N_-, t\mapsto \exp(tf_i)$ the $1$-parameter subgroup corresponding to $f_i$. Then for $w\in W$ and its reduced word $\bm{i}=(i_1,\dots, i_{\ell})$, there exists a map $y_{\bm{i}}\colon (\mathbb{C}^{\times})^{\ell}\to N_-^w$ given by 
\[
(t_1,\dots, t_{\ell})\mapsto \exp(t_1f_{i_1})\cdots\exp(t_{\ell}f_{i_{\ell}}).
\]
Then it is known that $y_{\bm{i}}$ is injective and its image is a Zariski open subset of $N_-^w$. See, for example, \cite[Proposition 2.18]{FZ}. The problem of finding explicit formulae for the inverse birational map $y_{\bm{i}}^{-1}$ is called \emph{the factorization problem}. By the way, an element $n\in N_-^w$ is \emph{totally positive} if and only if $n\in \Image y_{\bm{i}}$ and $y_{\bm{i}}^{-1}(n)\in \mathbb{R}_{>0}^{\ell}$ \cite[Proposition 2.7]{Lus:pos}. This problem is also formulated as follows: the map $y_{\bm{i}}$ induces an embedding of algebras 
\begin{align}
y_{\bm{i}}^{\ast}\colon\mathbb{C}[N_-^w]\to \mathbb{C}[ (\mathbb{C}^{\times})^{\ell}]\simeq \mathbb{C}[t_1^{\pm 1},\dots, t_{\ell}^{\pm1}].\label{clFeigin}
\end{align}
The problem is to describe each $t_{k}$ ($k=1,\dots, \ell$) as a rational function on $N_-^w$ explicitly. As mentioned in subsection \ref{introabst}, a solution of this problem is the Chamber Ansatz formula. 
Let $\varpi_i\in \Hom_{\rm alg\textrm{-}grp}(H, \mathbb{C}^{\times})$ be a fundamental weight corresponding to $i\in I$. Set $G_0:=N_-HN_+$ and, for $g\in G_0$, write the corresponding decomposition as $g=[g]_-[g]_0[g]_+$. For $i\in I$, denote by $\Delta_{\varpi_i, \varpi_i}$  the regular function on $G$ whose restriction to the open dense set $G_0$ is given by $\Delta_{\varpi_i, \varpi_i}(g):=\varpi_i([g]_0)$. Moreover, for $w_1, w_2\in W$, define $\Delta_{w_1\varpi_i, w_2\varpi_i}\in \mathbb{C}[G]$ by $\Delta_{w_1\varpi_i, w_2\varpi_i}(g)=\Delta_{\varpi_i, \varpi_i}(\overline{w}_1^{-1}g\overline{w}_2)$, where $\overline{w}_1$, $\overline{w}_2$ are specific lifts of $w_1$, $w_2$ to $N_{G}\left(H\right)$, respectively. These elements are called \emph{generalized minors}. 
Berenstein, Fomin and Zelevinsky introduced a biregular isomorphism $\eta_{w}\colon N_-^{w}\to N_-^{w}$ given by 
\[
n\mapsto [n^{T}\dot{w}]_-,
\]
here $n^T$ is a transpose of $n$ in $G$. This is called \emph{the twist automorphism on $N_-^w$}. Then the Chamber Ansatz formulae stand for the following description of $t_k$, $k=1,\dots, \ell$ \cite[Theorem 1.4]{BFZ}, \cite[Theorem 1.4]{BZ:Schu};
set $y:=y_{\bm{i}}(t_1,\dots,t_{\ell})$ and $w_{\leq m}:=s_{i_1}\cdots s_{i_m}$. Then,
\[
t_k=\frac{\prod_{j\in I\setminus\{i_k\}}\Delta_{w_{\le k}\varpi_{j}, \varpi_{j}}(\eta_{w}^{-1}(y))^{-a_{j, i_k}}}{\Delta_{w_{\leq k-1}\varpi_{i_k}, \varpi_{i_k}}(\eta_{w}^{-1}(y))\Delta_{w_{\leq k}\varpi_{i_k}, \varpi_{i_k}}(\eta_{w}^{-1}(y))},
\]
here $a_{ij}:=\langle h_i, \alpha_j\rangle$. 

\subsection{Quantum Chamber Ansatz formulae---Main result}A quantum analogue $\mathbf{A}_q[N_-^w]$ of the coordinate algebra $\mathbb{C}[N_-^{w}]$ is introduced in \cite{DP:qSch} and there are quantum analogues $[D_{w_1\varpi_{i}, w_2\varpi_i}]\in \mathbf{A}_q[N_-^w]$ of generalized minors on unipotent cells. There also exists a quantum analogue $\Phi_{\bm{i}}\colon \mathbf{A}_q[N_-^w]\to\mathcal{L}_{\bm{i}}$ of $y_{\bm{i}}^{\ast}$, which is known as the Feigin homomorphism \cite{Berenstein}. Here $\mathcal{L}_{\bm{i}}$ is a quantum torus in $\ell$-variables $t_1,\dots, t_{\ell}$. Moreover a quantum analogue $\eta_{w, q}\colon\mathbf{A}_q[N_-^w]\to\mathbf{A}_q[N_-^w]$ of the (dual of) the twist automorphism $\eta_w^{\ast}\colon \mathbb{C}[N_-^w]\to \mathbb{C}[N_-^w]$ is defined by Kimura and the author in  \cite{KO2}. Note that we do not have ``actual algebraic varieties'' but only have ``coordinate rings'' in the quantum settings. By using these materials, we obtain the Chamber Ansatz formula in the quantum settings.

\begin{theorem_i}[{Theorem \ref{t:monomial}, Corollary \ref{c:CA}}]
For $k=1,\dots, \ell$, we set 
\[
D_{w_{\leq k}\varpi_{i_{k}}, \varpi_{i_{k}}}^{\prime\; (\bm{i})}:=(\Phi_{\bm{i}}\circ \eta_{w, q}^{-1})([D_{w_{\leq k}\varpi_{i_{k}}, \varpi_{i_{k}}}]). 
\]
Then there exists an (explicit) integer $M_k$ such that 
\[
D_{w_{\leq k}\varpi_{i_{k}}, \varpi_{i_{k}}}^{\prime\; (\bm{i})}
=q^{M_k}t_{1}^{-d_1}t_{2}^{-d_2}\cdots t_{k}^{-d_k},
\]
where $d_j:=\langle w_{\leq j}h_{i_j}, w_{\leq k}\varpi_{i_k}\rangle$, $j=1,\dots, k$. These formulae deduce the following:
\[
t_k=q^{N_k}(D_{w_{\leq k-1}\varpi_{i_{k}}, \varpi_{i_{k}}}^{\prime\; (\bm{i})})^{-1}(D_{w_{\leq k}\varpi_{i_{k}}, \varpi_{i_{k}}}^{\prime\; (\bm{i})})^{-1}\dprod_{j\in I\setminus\{i_k\}}(D_{w_{\leq k}\varpi_{j}, \varpi_{j}}^{\prime\; (\bm{i})})^{-a_{j, i_k}}
\]
for some $N_k\in \mathbb{Z}$, here $\overrightarrow{\prod}$ stands for the ordered multiplication according to an arbitrarily fixed ordering on $I\setminus\{i_k\}$. 
\end{theorem_i}
This is a generalization of Berenstein-Rupel's result \cite[Corollary 1.2]{BR}. By this theorem, we can say that the quantum twist automorhism $\eta_{w, q}$ is a generalization of Berenstein-Rupel's quantum twist automorphism \cite[Theorem 2.10]{BR}, which has been constructed in the case that $w$ is the square of an acyclic Coxeter element, and corresponds to the one proposed in \cite[Conjecture 2.12 (c), Conjecture 6.20]{BR}. Moreover, its compatibility with the dual canonical basis, which is proved in \cite[Theorem 6.1]{KO2}, corresponds to Berenstein-Rupel's conjectural property  \cite[Conjecture  2.17 (a)]{BR}. See also Remark \ref{r:BR} below. 

The above theorem states the non-trivial monomiality of $(\Phi_{\bm{i}}\circ \eta_{q, w}^{-1})([D_{w_{\leq k}\varpi_{i_{k}}, \varpi_{i_{k}}}])$. In the appendix, we explain an explicit relation between this monomiality and the similar one appearing in the context of Cauchon-Goodearl-Letzter (CGL) extensions \cite{GeY, LeY}. It also would be interesting to understand this monomiality via categorifications. Actually, in non-quantum settings, Gei\ss-Leclerc-Schr\"{o}er have obtained an explanation by using their additive categorification \cite[Theorem 1, Theorem 2]{GLS:chamb}.

\subsection{Notation}The following are general notations in this paper.
\begin{itemize}
\item[(1)] For a $k$-algebra $\mathscr{A}$, we set $[a_{1},a_{2}]:=a_{1}a_{2}-a_{2}a_{1}$ for $a_{1},a_{2}\in\mathscr{A}$. \emph{An Ore set $\mathscr{M}$ of $\mathscr{A}$} stands for a left and right Ore set consisting of regular elements. Denote by $\mathscr{A}[\mathscr{M}^{-1}]$ the algebra of fractions with respect to the Ore set $\mathscr{M}$. In this case, $\mathscr{A}$ is naturally a subalgebra of $\mathscr{A}[\mathscr{M}^{-1}]$. See \cite[Chapter 6]{GoodWar} for more details.
\item[(2)] \emph{An $\mathscr{A}$-module $V$} means a left $\mathscr{A}$-module. The action of $\mathscr{A}$ on $V$ is denoted by $a.u$ for $a\in\mathscr{A}$ and $u\in V$.
\item[(3)] For two letters $i,j$, the symbol $\delta_{ij}$ stands for the Kronecker delta. 
\end{itemize}

\section{Preliminaries}
\subsection{Quantized enveloping algebras}
\begin{definition}
\label{d:rootdatum} A root datum consists of the following
data; 
\begin{enumerate}
\item $I$ : a finite index set, 
\item $\mathfrak{h}$ : a finite dimensional $\mathbb{Q}$-vector space, 
\item $P\subset\mathfrak{h}^{*}$ : a lattice, called \emph{the weight lattice}, 
\item $P^{\ast}=\left\{ h\in\mathfrak{h}\mid\left\langle h,P\right\rangle \subset\mathbb{Z}\right\}$,
called \emph{the coweight lattice}, 
with the canonical pairing $\left\langle \;,\;\right\rangle \colon P^{\ast}\times P\to\mathbb{Z}$, 
\item $\left\{ \alpha_{i}\right\} _{i\in I}\subset P$ : a subset, called the set of \emph{simple roots}, 
\item $\left\{ h_{i}\right\} _{i\in I}\subset P^{\ast}$ : a subset, called the set of \emph{simple coroots}, 
\item $\left(\;,\;\right)\colon P\times P\to\mathbb{Q}$ : a $\mathbb{Q}$-valued symmetric $\mathbb{Z}$-bilinear form on $P$. 
\end{enumerate}
satisfying the following conditions: 
\begin{enumerate}
\item[(a)] $\left(\alpha_{i},\alpha_{i}\right)\in2\mathbb{Z}_{>0}$ for $i\in I$, 
\item[(b)] $\left\langle h_{i},\lambda\right\rangle =2\left(\alpha_{i},\lambda\right)/\left(\alpha_{i},\alpha_{i}\right)$
for $\lambda\in P$ and $i\in I$, 
\item[(c)] $A=(a_{ij})_{i,j\in I}:=\left(\left\langle h_{i},\alpha_{j}\right\rangle \right)_{i,j\in I}$
is a symmetrizable generalized Cartan matrix, that is $\left\langle h_{i}, \alpha_{i} \right\rangle =2$, $\left\langle h_{i}, \alpha_{j} \right\rangle \in\mathbb{Z}_{\leq0}$ for $i\neq j$ and, $\left\langle h_{i},\alpha_{j}\right\rangle =0$ is equivalent to $\left\langle h_{j},\alpha_{i}\right\rangle =0$, 
\item[(d)] $\left\{ \alpha_{i}\right\} _{i\in I}\subset\mathfrak{h}^{*}$, $\left\{ h_{i}\right\} _{i\in I}\subset\mathfrak{h}$
are linearly independent subsets. 
\end{enumerate}
The $\mathbb{Z}$-submodule $Q=\sum_{i\in I}\mathbb{Z}\alpha_{i}\subset P$ is called \emph{the root lattice}. We set $Q_{+}=\sum\mathbb{Z}_{\geq0}\alpha_{i}\subset Q$ and $Q_{-}=-Q_{+}$. Let $P_{+}:=\left\{ \lambda\in P\mid\left\langle h_{i},\lambda\right\rangle \in\mathbb{Z}_{\geq0}\;\text{for all}\;i\in I\right\}$ and we assume that there exists $\left\{ \varpi_{i}\right\} _{i\in I}\subset P_{+}$ such that $\left\langle h_{i},\varpi_{j}\right\rangle =\delta_{ij}$. An element of $P_+$ is called \emph{a dominant integral weight}. 
\end{definition}
\begin{definition}
\label{d:Weyl} Let $W$ be \emph{the Weyl group} associated with the above root datum, that is, the group generated by $\{s_{i}\}_{i\in I}$ with the defining relations $s_i^2=e$ for $i\in I$ and $(s_is_j)^{m_{ij}}=e$ for $i, j\in I$, $i\neq j$. Here $e$ is the unit of $W$, $m_{ij}=2$ (resp.~$3, 4, 6, \infty$) if $a_{ij}a_{ji}=0$ (resp.~$1, 2, 3, \geq 4$), and $w^{\infty}:=e$ for any $w\in W$. We have the group homomorphisms $W\to\Aut\mathfrak{h}$ and $W\to\Aut\mathfrak{h}^{\ast}$ given by
\begin{align*}
s_{i}\left(h\right)&=h-\langle h, \alpha_{i}\rangle h_{i}&s_{i}\left(\mu\right)&=\mu-\langle h_{i},\mu\rangle\alpha_{i}
\end{align*}
for $h\in\mathfrak{h}$ and $\mu\in\mathfrak{h}^{\ast}$. For an element $w$ of $W$, $\ell(w)$ denotes the length of $w$, that is, the smallest integer $\ell$ such that there exist $i_{1},\dots,i_{\ell}\in I$
with $w=s_{i_{1}}\cdots s_{i_{\ell}}$. For $w\in W$, set 
\[
I(w):=\{\bm{i}=(i_{1},\dots,i_{\ell(w)})\in I^{\ell(w)}\mid w=s_{i_{1}}\cdots s_{i_{\ell(w)}}\}.
\]
An element of $I(w)$ is called \emph{a reduced word of $w$}. When fixing a reduced word $\bm{i}=(i_1,\dots, i_{\ell})\in I(w)$, we write $w_{\leq k}:=s_{i_1}\cdots s_{i_k}$ and $w_{k\leq}:=s_{i_k}\cdots s_{i_{\ell}}$ for $1\leq k\leq \ell$. Moreover set $w_{\leq 0}:=e$.
\end{definition}
\begin{notation} \label{n:indeterminate} Let $q$ be an indeterminate. Set 
\[
\begin{array}{l}
q_{i}:=q^{\frac{(\alpha_{i},\alpha_{i})}{2}},\ {\displaystyle [n]:=\frac{q^{n}-q^{-n}}{q-q^{-1}}\ \text{for\ }n\in\mathbb{Z},}\\
{\displaystyle \left[\begin{array}{c}
n\\
k
\end{array}\right]:=\begin{cases}
{\displaystyle \frac{[n][n-1]\cdots[n-k+1]}{[k][k-1]\cdots[1]}} & \text{if\ }n\in\mathbb{Z},k\in\mathbb{Z}_{>0},\\
1 & \text{if\ }n\in\mathbb{Z},k=0,
\end{cases}}\\
{\displaystyle [n]!:=[n][n-1]\cdots[1]\text{\ for\ }n\in\mathbb{Z}_{>0},[0]!:=1.}
\end{array}
\]
Note that $[n], \left[ \begin{array}{c} n\\k \end{array} \right]\in \mathbb{Z}[q^{\pm 1}]$ and $\displaystyle \left[ \begin{array}{c} n\\k \end{array} \right]=\displaystyle \frac{[n]!}{[k]![n-k]!}\ \text{if\ } n\geq k\geq 0$. For a rational function $R\in\mathbb{Q}(q)$, we define $R_{i}$ as the rational function obtained from $R$ by substituting $q$ by $q_{i}$. 
\end{notation} 
\begin{definition}
\label{d:QEA} \emph{The quantized enveloping algebra} $\mathbf{U}_q$ is the unital associative $\mathbb{Q}(q)$-algebra (associated
with $\left(P,I, \{ \alpha_{i}\} _{i\in I},\{ h_{i}\} _{i\in I},(\;,\;)\right)$) defined by the
generators 
\[
e_{i},f_{i}\;(i\in I),q^{h}\;(h\in P^{*}),
\]
and the relations (i)--(iv) below: 
\begin{enumerate}
\item[(i)] $q^{0}=1,\;q^{h}q^{h'}=q^{h+h'}$ for $h,h'\in P^{*}$,
\item[(ii)] $q^{h}e_{i}=q^{\langle h,\alpha_{i}\rangle}e_{i}q^{h},\;q^{h}f_{i}=q^{-\langle h,\alpha_{i}\rangle}f_{i}q^{h}$
for $h\in P^{*},i\in I$,
\item[(iii)] ${\displaystyle \left[e_{i},f_{j}\right]=\delta_{ij}\frac{k_{i}-k_{i}^{-1}}{q_{i}-q_{i}^{-1}}}$
for $i,j\in I$ where $k_{i}:=q^{\frac{(\alpha_{i},\alpha_{i})}{2}h_{i}}$,
\item[(iv)] ${\displaystyle \sum_{k=0}^{1-a_{ij}}(-1)^{k}\left[\begin{array}{c}
1-a_{ij}\\
k
\end{array}\right]_{i}x_{i}^{k}x_{j}x_{i}^{1-a_{ij}-k}=0}$ for $i,j\in I$ with $i\neq j$, and $x=e,f$. 
\end{enumerate}
The $\mathbb{Q}(q)$-subalgebra of $\mathbf{U}_q$ generated by $\{f_{i}\}_{i\in I}$ is denoted by $\mathbf{U}_q^{-}$. For $\alpha\in Q$, write $(\mathbf{U}_q)_{\alpha}:=\{x\in\mathbf{U}_q\mid q^{h}xq^{-h}=q^{\langle h, \alpha\rangle}x\;\text{for all}\;h\in P^{*}\}$.
The elements of $(\mathbf{U}_q)_{\alpha}$ are said to be \emph{homogeneous}. For
a homogeneous element $x\in(\mathbf{U}_q)_{\alpha}$, we set $\weight x=\alpha$. For any subset $X\subset\mathbf{U}_q$ and $\alpha\in Q$, we set $X_{\alpha}:=X\cap(\mathbf{U}_q)_{\alpha}$.
\end{definition}
\begin{definition}
\label{d:autom} Let $\vee\colon\mathbf{U}_q\to\mathbf{U}_q$ be the $\mathbb{Q}(q)$-algebra involution defined by 
\begin{align*}
e_{i}^{\vee} & =f_{i}, &  & f_{i}^{\vee}=e_{i}, &  & (q^{h})^{\vee}=q^{-h}.
\end{align*}
Let $\overline{\phantom{x}}\colon\mathbb{Q}(q)\to\mathbb{Q}(q)$,
$\overline{\phantom{x}}\colon\mathbf{U}_q\to\mathbf{U}_q$ be the $\mathbb{Q}$-algebra
involutions defined by 
\begin{align*}
\overline{q}=q^{-1}, &  & \overline{e_{i}}=e_{i}, &  & \overline{f_{i}}=f_{i}, &  & \overline{q^{h}} & =q^{-h}.
\end{align*}
Let $*,\varphi\colon\mathbf{U}_q\to\mathbf{U}_q$ be the $\mathbb{Q}\left(q\right)$-anti-algebra
involutions defined by 
\begin{align*}
*(e_{i}) & =e_{i}, & *(f_{i}) & =f_{i}, & *(q^{h}) & =q^{-h},\\
\varphi\left(e_{i}\right) & =f_{i}, & \varphi(f_{i}) & =e_{i}, & \varphi(q^{h}) & =q^{h}.
\end{align*}
Note that $\varphi=\vee\circ\ast=\ast\circ\vee$. 
\end{definition}
\begin{definition}\label{d:Lusform} Define the $\mathbb{Q}(q)$-bilinear form $(\ ,\ )_{L}\colon\mathbf{U}_q^{-}\times\mathbf{U}_q^{-}\to\mathbb{Q}(q)$ as follows. See, for example, \cite[Chapter 1]{Lus:intro} for more details: for $i\in I$, there uniquely exist the $\mathbb{Q}(q)$-linear
maps $e'_{i}$, $_{i}e'\colon\mathbf{U}_q^{-}\to\mathbf{U}_q^{-}$ satisfying
\[
\begin{array}{lc}
e'_{i}\left(xy\right)=e'_{i}\left(x\right)y+q_{i}^{\langle h_{i}, \weight x\rangle}xe'_{i}\left(y\right), & e'_{i}(f_{j})=\delta_{ij},\\
_{i}e'\left(xy\right)=q_{i}^{\langle h_{i}, \weight y\rangle}{_{i}e'}\left(x\right)y+x\;{_{i}e'}\left(y\right), & _{i}e'(f_{j})=\delta_{ij}
\end{array}
\]
for homogeneous elements $x,y\in\mathbf{U}_q^{-}$. Then there uniquely exists the symmetric $\mathbb{Q}(q)$-bilinear form satisfying
\begin{align}
(1,1)_{L}&=1,  &   (f_{i}x,y)_{L}&=\frac{1}{1-q_{i}^{2}}(x,e'_{i}(y))_{L},   & (xf_{i},y)_{L}&=\frac{1}{1-q_{i}^{2}}(x,{_{i}e'}(y))_{L}\label{eq:adjoint}
\end{align}
for $x, y\in \mathbf{U}_q^-$. In fact, $(\ ,\ )_{L}$ is nondegenerate and it has the following property:
\[
\left(\ast(x),\ast(y)\right)_{L}=\left(x,y\right)_{L}
\]
for all $x,y\in\mathbf{U}_q^{-}$.
\end{definition}
\subsection{Lusztig's braid group symmetries}
We present the definition of braid group actions on integrable modules and quantized enveloping algebras, and review their fundamental properties. All statements in this subsections can be found in \cite{Lus:intro, Saito:PBW}. 
\begin{definition}\label{d:intmodules}
Let $V$ be a $\mathbf{U}_q$-module. For $\mu \in P$, we set
\[
V_{\mu}:= \{u \in V \mid q^h.u=q^{\langle h, \mu \rangle}u\ \text{\ for\ all\ } h \in P^{\ast}\}.
\]
This is called \emph{the weight space of $V$ of weight $\mu$}, and for $u\in V_{\mu}$, we write $\weight u:=\mu$. A $\mathbf{U}_q$-module $V=\bigoplus_{\mu\in P}V_{\mu}$ with weight space decomposition is said to be \emph{integrable} if $e_i$ and $f_i$ act locally nilpotently on $V$ for all $i\in I$. 
\end{definition}
\begin{definition}\label{d:repform}
For $\lambda\in P_{+}$, denote by $V(\lambda)$ the integrable highest weight $\mathbf{U}_q$-module generated by a highest weight vector $u_{\lambda}$ of weight $\lambda$. Note that $V(\lambda)$ is irreducible. There exists a unique $\mathbb{Q}(q)$-bilinear form $(\;,\;)_{\lambda}^{\varphi}\colon V(\lambda)\times V(\lambda)\to\mathbb{Q}(q)$ such that
\begin{align*}
 \left(u_{\lambda},u_{\lambda}\right)_{\lambda}^{\varphi} & =1 & (x.u_{1},u_{2})_{\lambda}^{\varphi} & =(u_{1},\varphi(x).u_{2})_{\lambda}^{\varphi}
\end{align*}
for $u_{1},u_{2}\in V(\lambda)$ and $x\in\mathbf{U}_q$. Moreover the form $(\;,\;)_{\lambda}^{\varphi}$ is nondegenerate and symmetric. There exists the $\mathbb{Q}$-linear automorphism $\overline{\phantom{x}}\colon V(\lambda)\to V(\lambda)$ given by $\overline{x.u_{\lambda}}=\overline{x}.u_{\lambda}$ for $x\in \mathbf{U}_q$.

For $w\in W$, define the element $u_{w\lambda}\in V(\lambda)$ by
\begin{align*}
u_{w\lambda}=f_{i_{1}}^{(\langle h_{i_{1}}, s_{i_{2}}\cdots s_{i_{\ell}}\lambda\rangle)}\cdots f_{i_{\ell-1}}^{(\langle h_{i_{\ell-1}}, s_{i_{\ell}}\lambda\rangle)}f_{i_{\ell}}^{(\langle h_{i_{\ell}}, \lambda\rangle)}.u_{\lambda}
\end{align*}
for $(i_{1},\dots,i_{\ell})\in I(w)$. It is known that this element does not depend on the choice of $(i_{1},\dots,i_{\ell})\in I(w)$ and $w\in W$. See, for example, \cite[Proposition 39.3.7]{Lus:intro}. Then $\left(u_{w\lambda},u_{w\lambda}\right)_{\lambda}^{\varphi}=1$ and $\overline{u_{w\lambda}}=u_{w\lambda}$. 
\end{definition}
\begin{definition}
\label{d:repbraid} Let $V=\bigoplus_{\mu\in P}V_{\mu}$ be an integrable $\mathbf{U}_q$-module. We can define a $\mathbb{Q}(q)$-linear automorphism $T_{i}\colon V\to V$ for $i\in I$ by 
\[
T_{i}(u):=\sum_{-a+b-c=\langle h_{i},\mu\rangle}(-1)^{b}q_{i}^{-ac+b}e_{i}^{(a)}f_{i}^{(b)}e_{i}^{(c)}.u
\]
for $u\in V_{\mu}$ and $\mu\in P$. 
\end{definition}
\begin{definition}\label{d:braidaction}
We can define a $\mathbb{Q}\left(q\right)$-algebra automorphism $T_{i}\colon\mathbf{U}_q\to\mathbf{U}_q$ for $i\in I$ by the following formulae: 
\begin{subequations}
\begin{align*}
T_{i}(q^{h}) & =q^{s_{i}\left(h\right)},\\
T_{i}\left(e_{j}\right) & =\begin{cases}
-f_{i}k_{i} & \text{for}\;j=i,\\
{\displaystyle \sum_{r+s=-\left\langle h_{i},\alpha_{j}\right\rangle }\left(-1\right)^{r}q_{i}^{-r}e_{i}^{\left(s\right)}e_{j}e_{i}^{\left(r\right)}} & \text{for}\;j\neq i,
\end{cases}\\
T_{i}\left(f_{j}\right) & =\begin{cases}
-k_{i}^{-1}e_{i} & \text{for}\;j=i,\\
{\displaystyle \sum_{r+s=-\left\langle h_{i},\alpha_{j}\right\rangle }\left(-1\right)^{r}q_{i}^{r}f_{i}^{\left(r\right)}f_{j}f_{i}^{\left(s\right)}} & \text{for}\;j\neq i.
\end{cases}
\end{align*}
\end{subequations} 
\end{definition}
The following are fundamental properties of $T_i$. 
\begin{proposition}\label{p:braid} 
Let $V$ be an integrable $\mathbf{U}_q$-module. 
\begin{enumerate}
\item[(1)] For  $i\in I$, $T_{i}(x.u)=T_{i}(x).T_{i}(u)$ for $u\in V$ and $x\in \mathbf{U}_q$.

\item[(2)] For $w\in W$, the composition maps $T_{w}:=T_{i_{1}}\cdots T_{i_{\ell}}\colon V\to V$, $\mathbf{U}_q\to \mathbf{U}_q$ do not depend on the choice of $(i_{1},\dots,i_{\ell})\in I(w)$. 

\item[(3)] For $\mu\in P$ and $w\in W$, $T_{w}$ maps $V_{\mu}$ to $V_{w\mu}$. 
\end{enumerate}
\end{proposition}
\begin{proposition}\label{p:extremal} 
Let $V$ be an integrable $\mathbf{U}_q$-module and $i\in I$. Then, for $u\in V_{\mu}\cap \Ker (e_i.)$ and $u'\in V_{\mu'}\cap \Ker (f_i.)$, we have
\begin{align*}
T_i^{-1}(u)&=f_i^{(\langle h_i, \mu\rangle)}.u& T_i(u')&=e_i^{(-\langle h_i, \mu'\rangle)}.u'.
\end{align*}
In particular, for $\lambda\in P_{+}$ and $w\in W$, we have 
\begin{align*}
u_{w\lambda}=(T_{w^{-1}})^{-1}(u_{\lambda}).
\end{align*}
\end{proposition}
\begin{proposition}\label{p:kernel}

\textup{(1)} For $i\in I$, we have $\Ker e'_{i}=\mathbf{U}_q^{-}\cap T_{i}\mathbf{U}_q^{-}$ and $\Ker {_ie'}=\mathbf{U}_q^{-}\cap T_{i}^{-1}\mathbf{U}_q^{-}$. 

\textup{(2)} For $i\in I$ and $x,y\in\Ker e'_{i}$, we have $\left(x,y\right)_{L}=\left(T_{i}^{-1}(x),T_{i}^{-1}(y)\right)_{L}$.
\end{proposition}
\subsection{Canonical/Dual canonical bases}
Canonical bases($=$lower global bases) are defined by Lusztig \cite{Lus:can1,Lus:quiper,Lus:intro} and Kashiwara \cite{Kas:Qana} independently. In this subsection, we briefly review the definitions of canonical bases of $\mathbf{U}_q^-$ and $V(\lambda)$, $\lambda\in P_+$, following Kashiwara \cite{Kas:Qana}. Let $\mathcal{A}_0$ be the subalgebra of $\mathbb{Q}(q)$ consisting of rational functions without poles at $q=0$. Set $\mathcal{A}:=\mathbb{Q}[q^{\pm 1}]$. 
\begin{definition}\label{d:canonicalbases}
For $i\in I$, we have $\mathbf{U}_q^-=\bigoplus_{k\in \mathbb{Z}_{\geq 0}}f_i^{(k)}\Ker e'_i$ \cite[3.5]{Kas:Qana}. Hence we can define the $\mathbb{Q}(q)$-linear maps $\tilde{e}_i$, $\tilde{f}_i\colon \mathbf{U}_q^-\to \mathbf{U}_q^-$ by
\begin{align*}
\tilde{e}_i(f_i^{(k)}u)=f_i^{(k-1)}u&&& \tilde{f}_i(f_i^{(k)}u)=f_i^{(k+1)}u
\end{align*}
for $u\in \Ker e'_i$ where $f_i^{(-1)}u:=0$. We have $\tilde{e}_i\circ \tilde{f}_i=\mathrm{id}_{\mathbf{U}_q^-}$. Set
\begin{align*}
\mathscr{L}(\infty)&:=\sum_{\ell\geq 0, i_1,\dots , i_{\ell}\in I}\mathcal{A}_0\tilde{f}_{i_1}\cdots\tilde{f}_{i_{\ell}}1 \subset \mathbf{U}_q^-,\\
\mathscr{B}(\infty)&:=\{\tilde{f}_{i_1}\cdots\tilde{f}_{i_{\ell}}1\ \mathrm{mod}\ q\mathscr{L}(\infty)\mid \ell\geq 0, i_1,\dots, i_{\ell}\in I\} \subset \mathscr{L}(\infty)/q\mathscr{L}(\infty).
\end{align*}
Henceforth write $\tilde{b}_{\infty}:=1\ \mathrm{mod}\ q\mathscr{L}(\infty)$. The pair $(\mathscr{L}(\infty), \mathscr{B}(\infty))$ satisfies the following properties \cite[Theorem 4]{Kas:Qana}:
\begin{itemize}
\item[(i)] $\mathscr{L}(\infty)$ is a free $\mathcal{A}_0$-module and $\mathbb{Q}(q)\otimes_{\mathcal{A}_0}\mathscr{L}(\infty)\simeq \mathbf{U}_q^-$,
\item[(ii)] $\mathscr{B}(\infty)$ is a basis of the $\mathbb{Q}$-vector space $\mathscr{L}(\infty)/q\mathscr{L}(\infty)$,
\item[(iii)] $\tilde{e}_i\mathscr{L}(\infty)\subset \mathscr{L}(\infty)$ and $\tilde{f}_i\mathscr{L}(\infty)\subset \mathscr{L}(\infty)$ for all $i\in I$,
\item[(iv)] $\tilde{e}_i$ and $\tilde{f}_i$ induce $\tilde{e}_i\colon \mathscr{B}(\infty)\to \mathscr{B}(\infty)\coprod \{ 0\}$ and $\tilde{f}_i\colon \mathscr{B}(\infty)\to \mathscr{B}(\infty)$, respectively, for all $i\in I$,
\item[(v)] For $\tilde{b}\in \mathscr{B}(\infty)$ with $\tilde{e}_i\tilde{b}\in \mathscr{B}(\infty)$, we have $\tilde{b}=\tilde{f}_i\tilde{e}_i\tilde{b}$.
\end{itemize}
This pair $(\mathscr{L}(\infty), \mathscr{B}(\infty))$ is called \emph{the (lower) crystal basis of $\mathbf{U}_q^-$}. For $i\in I$, define the maps $\varepsilon_i$, $\varphi_i\colon \mathscr{B}(\infty)\to \mathbb{Z}$  by
\begin{align*}
\varepsilon_i(\tilde{b})=\max\{ k\in \mathbb{Z}_{\geq 0}\mid \tilde{e}_i^k \tilde{b}\neq 0\}&&& \varphi_i(\tilde{b})=\varepsilon_i(\tilde{b})+\langle h_i, \weight \tilde{b}\rangle,
\end{align*}
for $\tilde{b}\in \mathscr{B}(\infty)$. Then the sextuple $(\mathscr{B}(\infty); \weight, \{\tilde{e}_i\}_{i\in I}, \{\tilde{f}_i\}_{i\in I}, \{\varepsilon_i\}_{i\in I}, \{\varphi_i\}_{i\in I})$ is a crystal in the sense of \cite{Kas:Dem}.

Moreover we have $\ast(\mathscr{L}(\infty))=\mathscr{L}(\infty)$ and $\ast(\mathscr{B}(\infty))=\mathscr{B}(\infty)$ \cite[Proposition 5.2.4]{Kas:Qana}, \cite[Theorem 2.1.1]{Kas:Dem}. Hence we can define a new crystal $(\mathscr{B}(\infty); \weight, \{\tilde{e}_i^{\ast}\}_{i\in I}, \{\tilde{f}_i^{\ast}\}_{i\in I}, \{\varepsilon_i^{\ast}\}_{i\in I}, \{\varphi_i^{\ast}\}_{i\in I})$ by
\[
\varepsilon_i^{\ast}:=\varepsilon_i\circ\ast, \varphi_i^{\ast}:=\varphi_i\circ\ast, \tilde{e}_i^{\ast}:=\ast\circ\tilde{e}_i\circ\ast, \tilde{f}_i^{\ast}:=\ast\circ\tilde{f}_i\circ\ast. 
\]
Note that $\varepsilon_i^{\ast}(\tilde{b})=\max\{k\in \mathbb{Z}_{\geq 0}\mid (\tilde{e}_i^{\ast})^k\tilde{b}\neq 0\}$. 

Let $\mathbf{U}_{\mathcal{A}}^-$ be the $\mathcal{A}$-subalgebra of $\mathbf{U}_q^-$ generated by $\{f_i^{(k)}\}_{i\in I, k\in \mathbb{Z}_{\geq 0}}$. Then the canonical map
\begin{align*}
\mathscr{L}(\infty)\cap \overline{\mathscr{L}(\infty)}\cap\mathbf{U}_{\mathcal{A}}^-\to \mathscr{L}(\infty)/q\mathscr{L}(\infty)
\end{align*}
is an isomorphism of $\mathbb{Q}$-vector spaces \cite[Theorem 6]{Kas:Qana}. The inverse of this map is denoted by $\Glow$. The set $\mathbf{B}^{\mathrm{low}}:=\{\Glow(\tilde{b})\}_{\tilde{b}\in \mathscr{B}(\infty)}$ is an $\mathcal{A}$-basis of $\mathbf{U}_{\mathcal{A}}^-$ and this is called \emph{the canonical basis} of $\mathbf{U}_q^-$. We have $\ast(\Glow(\tilde{b}))=\Glow(\ast \tilde{b})$ for $\tilde{b}\in \mathscr{B}(\infty)$. 
\end{definition}
\begin{definition}\label{crystalsinrepresentation}
Let $\lambda\in P_+$. For $i\in I$, we define the $\mathbb{Q}(q)$-linear maps $\tilde{e}_i$, $\tilde{f}_i\colon V(\lambda)\to V(\lambda)$ by
\begin{align*}
\tilde{e}_i(f_i^{(k)}.u)=f_i^{(k-1)}.u&&&\tilde{f}_i(f_i^{(k)}.u)=f_i^{(k+1)}.u
\end{align*}
for $u\in \Ker (e_i.) \cap V(\lambda)$, where $f_i^{(-1)}.u:=0$. Set
\begin{align*}
\mathscr{L}(\lambda)&:=\sum_{\ell\geq 0, i_1,\dots, i_{\ell}\in I}\mathcal{A}_0\tilde{f}_{i_1}\cdots\tilde{f}_{i_{\ell}}u_{\lambda} \subset V(\lambda),\\
\mathscr{B}(\lambda)&:=\{\tilde{f}_{i_1}\cdots\tilde{f}_{i_{\ell}}u_{\lambda}\ \mathrm{mod}\ q\mathscr{L}(\lambda)\mid \ell\geq 0, i_1,\dots, i_{\ell}\in I\}\setminus \{ 0 \} \subset \mathscr{L}(\lambda)/q\mathscr{L}(\lambda).
\end{align*}
Henceforth write $b_{\lambda}:=u_{\lambda}\ \mathrm{mod}\ q\mathscr{L}(\lambda)\in \mathscr{B}(\lambda)$. Then the pair $(\mathscr{L}(\lambda), \mathscr{B}(\lambda))$ satisfies the following properties \cite[Theorem 2]{Kas:Qana}:
\begin{itemize}
\item[(i)] $\mathscr{L}(\lambda)$ is a free $\mathcal{A}_0$-module and $\mathbb{Q}(q)\otimes_{\mathcal{A}_0}\mathscr{L}(\lambda)\simeq V(\lambda)$,
\item[(ii)] $\mathscr{B}(\lambda)$ is a basis of the $\mathbb{Q}$-vector space $\mathscr{L}(\lambda)/q\mathscr{L}(\lambda)$,
\item[(iii)] $\tilde{e}_i\mathscr{L}(\lambda)\subset \mathscr{L}(\lambda)$ and $\tilde{f}_i\mathscr{L}(\lambda)\subset \mathscr{L}(\lambda)$ for all $i\in I$,
\item[(iv)] $\tilde{e}_i$ and $\tilde{f}_i$ induce $\tilde{e}_i\colon \mathscr{B}(\lambda)\to \mathscr{B}(\lambda)\coprod \{ 0\}$ and $\tilde{f}_i\colon \mathscr{B}(\lambda)\to \mathscr{B}(\lambda)\coprod \{ 0\}$, respectively, for all $i\in I$,
\item[(v)] For $b, b'\in \mathscr{B}(\lambda)$, we have $b'=\tilde{f}_ib$ if and only if $b=\tilde{e}_ib'$.
\end{itemize}
This pair $(\mathscr{L}(\lambda), \mathscr{B}(\lambda))$ is called \emph{the (lower) crystal basis of $V(\lambda)$}. For $i\in I$, define the maps $\varepsilon_i$, $\varphi_i\colon \mathscr{B}(\lambda)\to \mathbb{Z}$ by
\begin{align*}
\varepsilon_i(b)=\max\{k\in \mathbb{Z}_{\geq 0}\mid \tilde{e}_i^k b\neq 0\}&&&\varphi_i(b)=\max\{ k\in \mathbb{Z}_{\geq 0}\mid \tilde{f}_i^k b\neq 0\}=\varepsilon_i(b)+\langle h_i, \weight b\rangle,
\end{align*}
for $b\in \mathscr{B}(\lambda)$. Then the sextuple $(\mathscr{B}(\lambda); \weight, \{\tilde{e}_i\}_{i\in I}, \{\tilde{f}_i\}_{i\in I}, \{ \varepsilon_i\}_{i\in I}, \{\varphi_i\}_{i\in I})$ is a crystal.

Set $V_{\mathcal{A}}(\lambda):=\mathbf{U}_{\mathcal{A}}^-.u_{\lambda}$. Then the canonical map 
\begin{align*}
\mathscr{L}(\lambda)\cap \overline{\mathscr{L}(\lambda)}\cap V_{\mathcal{A}}(\lambda)\to \mathscr{L}(\lambda)/q\mathscr{L}(\lambda)
\end{align*}
is an isomorphism of $\mathbb{Q}$-vector spaces \cite[Theorem 6]{Kas:Qana}. The inverse of this map is denoted by $\Glow_{\lambda}$. The set $\mathbf{B}^{\mathrm{low}}(\lambda):=\{\Glow_{\lambda}(b)\}_{b\in \mathscr{B}(\lambda)}$ is an $\mathcal{A}$-basis of $V_{\mathcal{A}}(\lambda)$ and this is called \emph{the canonical basis of $V(\lambda)$}. For $b\in \mathscr{B}(\lambda)$, write
\begin{align*}
\tilde{e}_i^{\mathrm{max}}b:=\tilde{e}_i^{\varepsilon_i(b)}b&&&\tilde{f}_i^{\mathrm{max}}b:=\tilde{f}_i^{\varphi_i(b)}b.
\end{align*}
\end{definition}
\begin{definition}\label{d:dualcan}
Denote by $\mathbf{B}^{\mathrm{up}}$ (resp.~$\mathbf{B}^{\mathrm{up}}(\lambda)$, $\lambda\in P_+$) the basis of $\mathbf{U}_q^{-}$ (resp.~$V(\lambda)$) dual to $\mathbf{B}^{\mathrm{low}}$
(resp.~$\mathbf{B}^{\mathrm{low}}(\lambda)$) with respect to the bilinear form $(\ ,\ )_{L}$ (resp.~$(\;,\;)_{\lambda}^{\varphi}$), that is, $\mathbf{B}^{\mathrm{up}}=\{\Gup(\tilde{b})\}_{\tilde{b}\in\mathscr{B}(\infty)}$ (resp.~$\mathbf{B}^{\mathrm{up}}(\lambda)=\{\Gup(b)\}_{b\in\mathscr{B}(\lambda)}$) such that 
\begin{align*}
(G(\tilde{b}),\Gup(\tilde{b}'))_{L}=\delta_{\tilde{b},\tilde{b}'}&&&
(\text{resp.~} (\Glow_{\lambda}(b),\Gup_{\lambda}(b'))_{\lambda}^{\varphi}=\delta_{b,b'})
\end{align*}
for any $\tilde{b},\tilde{b}'\in\mathscr{B}(\infty)$ (resp.~$b,b'\in\mathscr{B}(\lambda)$). 
\end{definition}
\begin{example}\label{e:ext}
For $\lambda\in P_+$ and $w\in W$, the vector $u_{w\lambda}$ belongs to $\mathbf{B}^{\mathrm{low}}(\lambda)$ and $\mathbf{B}^{\mathrm{up}}(\lambda)$. 
\end{example}
\begin{proposition}[{\cite[Lemma 5.1.1]{Kasglo}}]\label{p:EFaction}
For $i\in I$, $\lambda\in P_+$ and $b\in \mathscr{B}(\lambda)$, we have
\begin{align*}
e_i^{(\varepsilon_i(b))}.\Gup_{\lambda}(b)&=\Gup_{\lambda}(\tilde{e}_i^{\mathrm{max}}b)& e_i^{(k)}.\Gup_{\lambda}(b)&=0\ \text{if}\ k>\varepsilon_i(b),\\
f_i^{(\varphi_i(b))}.\Gup_{\lambda}(b)&=\Gup_{\lambda}(\tilde{f}_i^{\mathrm{max}}b)& f_i^{(k)}.\Gup_{\lambda}(b)&=0\ \text{if}\ k>\varphi_i(b).
\end{align*}
\end{proposition}
\subsection{Unipotent quantum matrix coefficients}
We present the definition and the properties of quantum analogues of matrix coefficients on unipotent groups. The dual canonical basis elements of $\mathbf{U}_q^-$ are described as the quantum matrix coefficients associated with dual canonical basis elements of integrable highest weight modules. 
\begin{definition}\label{d:minor} For $\lambda\in P_{+}$ and $u,u'\in V(\lambda)$, define the element $D_{u,u'}\in\mathbf{U}_q^{-}$ by the following property:
\[
(D_{u,u'},x)_{L}=(u,x.u')_{\lambda}^{\varphi}
\]
for all $x\in\mathbf{U}_q^{-}$. Note that the element $D_{u, u'}$ is uniquely determined by the nondegeneracy of the pairing $(\;,\;)_L$. We call an element of this form \emph{a unipotent quantum matrix coefficient}. Note that $\weight\left(D_{u,u'}\right)=\weight u-\weight u'$ 
for weight vectors $u,u'\in V(\lambda)$. For $w,w'\in W$, write 
\[
D_{w\lambda,w'\lambda}:=D_{u_{w\lambda},u_{w'\lambda}}.
\] 
An element of this form is called \emph{a unipotent quantum
minor}. 
\end{definition}
The following property is nothing but the well-known ``compatibility'' between the canonical basis of $\mathbf{U}_q^-$ and that of $V(\lambda)$. The assertion (1) follows from \cite[Theorem 5]{Kas:Qana}, and the assertion (2) follows from \cite[Proposition 25.2.6]{Lus:intro} and \cite[8.2.2 (iii), (iv)]{Kas:modified}. See also \cite[Proposition 3.46]{KO2}. 
\begin{proposition}\label{p:unipminor}
Let $w\in W$, $\lambda\in P_+$ and $b\in\mathscr{B}(\lambda)$. Then we have the following: 
\begin{itemize}
\item[(1)] the element $D_{\Gup_{\lambda}(b), u_{\lambda}}$ belongs to $\mathbf{B}^{\mathrm{up}}$, 
\item[(2)] the element $D_{u_{w\lambda}, \Gup_{\lambda}(b)}$ belongs to $\mathbf{B}^{\mathrm{up}}$ or equals $0$.  
\end{itemize}
In particular, nonzero unipotent quantum minors are elements of $\mathbf{B}^{\mathrm{up}}$.
\end{proposition}
\begin{remark}
In \cite{KO2}, we write $D_{\Gup_{\lambda}\left(b\right),u_{\lambda}}=\Gup\left(\overline{\jmath}_{\lambda}(b)\right)$
 and $D_{u_{w\lambda},\Gup_{\lambda}\left(b\right)}=\Gup\left(\ast\overline{\jmath}_{w\lambda}^{\vee}\left(b\right)\right)$
for $b\in\mathscr{B}\left(\lambda\right)$ by using the maps $\overline{\jmath}_{\lambda}\colon \mathscr{B}\left(\lambda\right)\to \mathscr{B}(\infty)$ and $\overline{\jmath}_{w\lambda}^{\vee}\colon \mathscr{B}\left(\lambda\right)\to \mathscr{B}(\infty)\coprod\{0\}$. See also Remark \ref{r:Demazure} below. 
\end{remark}
The following slightly technical proposition is used when we consider the inverse of a quantum twist automorphism below (see \eqref{inverse}). Recall the notation in Definition \ref{d:Weyl}.
\begin{proposition}[{\cite[Corollary 6.4]{Kimura:qunip}, \cite[Theorem 3.48]{KO2}}]\label{p:replace}
Let $\lambda\in P_+$, $w\in W$ and fix $\bm{i}\in I(w)$. Then, for $0\leq k\leq \ell(w)$, there exist $\lambda'\in P_+$ and $b\in \mathscr{B}(\lambda')$ such that $D_{u_{w\lambda'}, \Gup_{\lambda'}(b)}=D_{w_{\leq k}\lambda, \lambda}$.
\end{proposition}
\subsection{Quantum unipotent cells and quantum twist automorphisms}
A quantum unipotent cell is a quantum analogue of the coordinate algebra of a unipotent cell. The quantum unipotent cells are essentially introduced by De Concini-Procesi \cite{DP:qSch}. We also define quantum twist automorphisms, which are introduced by Kimura and the author \cite{KO2}, on quantum unipotent cells. They are the dramatis personae of the Chamber Ansatz formulae.  
\begin{proposition}[{{\cite[Proposition 3.2.3, 3.2.5]{Kas:Dem}}}]
\label{p:Demazure} For $\lambda\in P_{+}$, $w\in W$ and $\bm{i}=\left(i_{1},\cdots,i_{\ell}\right)\in I\left(w\right)$,
we set 
\[
\mathscr{B}_{w}\left(\lambda\right):=\left\{ \widetilde{f}_{i_{1}}^{a_{1}}\cdots\widetilde{f}_{i_{\ell}}^{a_{\ell}}b_{\lambda}\mid \bm{a}=\left(a_{1},\cdots,a_{\ell}\right)\in\mathbb{Z}_{\geq0}^{\ell}\right\} \setminus\left\{ 0\right\} \subset\mathscr{B}\left(\lambda\right).
\]
Then we have 
\[
V_{w}(\lambda):=\mathbf{U}_q^{+}.u_{w\lambda}=\sum_{b\in\mathscr{B}_{w}\left(\lambda\right)}\mathbb{Q}\left(q\right)\Glow_{\lambda}\left(b\right).
\]
This $\mathbf{U}_q^{\geq 0}$-module $V_w(\lambda)$ is called \emph{a Demazure module}.

\textup{(2)} For $w\in W$ and $\bm{i}=\left(i_{1},\cdots,i_{\ell}\right)\in I\left(w\right)$,
we set 
\[
\mathscr{B}_{w}\left(\infty\right)=\left\{ \widetilde{f}_{i_{1}}^{a_{1}}\cdots\widetilde{f}_{i_{\ell}}^{a_{\ell}}\tilde{b}_{\infty}\mid \bm{a}=\left(a_{1},\cdots,a_{\ell}\right)\in\mathbb{Z}_{\geq0}^{\ell}\right\} 
\]
and $\mathbf{U}_{w}^{-}:=\sum_{a_{1},\dots,a_{\ell}\in\mathbb{Z}_{\geq 0}}\mathbb{Q}\left(q\right)f_{i_{1}}^{a_{1}}\cdots f_{i_{\ell}}^{a_{\ell}}$.
Then we have 
\[
\mathbf{U}_{w}^{-}=\sum_{\tilde{b}\in\mathscr{B}_{w}\left(\infty\right)}\mathbb{Q}\left(q\right)\Glow(\tilde{b}).
\]
\end{proposition}
For more details on Demazure modules and their crystal bases, see Kashiwara
\cite{Kas:Dem}. 
\begin{remark}\label{r:Demazure}
Recall Proposition \ref{p:unipminor}. If $b\in \mathscr{B}_{w}\left(\lambda\right)$, then $D_{\Gup_{\lambda}(b), u_{\lambda}}=\Gup(\tilde{b})$ for some $\tilde{b}\in \mathscr{B}_w(\infty)$. The element $D_{u_{w\lambda}, \Gup_{\lambda}(b)}$ is equal to $0$ if and only if $b\notin \mathscr{B}_w(\lambda)$.  
\end{remark}
\begin{definition}
\label{d:qclosed} Let $w\in W$. Set 
\[
\left(\mathbf{U}_{w}^{-}\right)^{\perp}:=\{x\in\mathbf{U}_q^{-}\mid(x,\mathbf{U}_{w}^{-})_{L}=0\}.
\]
Then, by the property of the pairing $(\;,\;)_L$, $\left(\mathbf{U}_{w}^{-}\right)^{\perp}$ is a two-sided ideal of $\mathbf{U}_q^{-}$. Hence we obtain a $\mathbb{Q}(q)$-algebra $\mathbf{U}_q^{-}/\left(\mathbf{U}_{w}^{-}\right)^{\perp}$, which is denoted by $\mathbf{A}_q[N_{-}\cap X_w]$ and called \emph{a quantum closed unipotent cell}. See \cite{KO2} for the meaning of the notation. 

The quantum closed unipotent cell has a $Q_{-}$-graded algebra structure induced from that of $\mathbf{U}_q^{-}$. Describe the canonical projection $\mathbf{U}_q^{-}\to\mathbf{A}_q[N_{-}\cap X_w]$ as $x\mapsto[x]$. The element $[x]$ clearly depends on $w$, however, we omit to write $w$ because it will cause no confusion below. By Proposition \ref{p:Demazure}, we have 
\[
\left(\mathbf{U}_{w}^{-}\right)^{\perp}=\sum_{\tilde{b}\in\mathscr{B}\left(\infty\right)\setminus\mathscr{B}_{w}\left(\infty\right)}\mathbb{Q}\left(q\right)\Gup(\tilde{b}).
\]
Hence $\mathbf{A}_q[N_{-}\cap X_w]$ has the dual canonical basis $\{[\Gup(\tilde{b})]\mid \tilde{b}\in \mathscr{B}_w(\infty)\}$. 
\end{definition}
The following multiplicative property and $q$-central property of unipotent quantum minors are well-known. We should note that the explicit powers of $q$ in the following formulae depend on the definitions of unipotent quantum minors delicately (cf.~\cite{Yak:inv, GeY}). A slightly detailed treatment in the same convention as ours can be found in \cite{KO2}. 
\begin{proposition}[{\cite[subsection 3.8]{Yak:inv}, \cite[section 6]{Kimura:qunip}}]
\label{p:comm} Let $w\in W$ and set $\mathcal{D}_{w}:=\{q^{m}D_{w\lambda,\lambda}\mid m\in \mathbb{Z}, \lambda\in P_{+}\}$.
Then the set $[\mathcal{D}_{w}]$ is an Ore set of $\mathbf{A}_q[N_{-}\cap X_w]$ consisting of $q$-central elements. More explicitly, for $\lambda,\lambda'\in P_{+}$ and a homogeneous element $[x]\in\mathbf{A}_q[N_{-}\cap X_w]$, we have 
\begin{itemize}
\item[(1)] $q^{-(\lambda,w\lambda'-\lambda')}D_{w\lambda,\lambda}D_{w\lambda',\lambda'} =D_{w(\lambda+\lambda'),\lambda+\lambda'}$ in $\mathbf{U}_q^-$, 
\item[(2)] $[D_{w\lambda,\lambda}][x]  =q^{(\lambda+w\lambda,\weight x)}[x][D_{w\lambda,\lambda}]$ in $\mathbf{A}_q[N_{-}\cap X_w]$.
\end{itemize}
\end{proposition}
\begin{definition}\label{d:localization} By Proposition \ref{p:comm}, we can consider
the following localization: 
\begin{align*}
\mathbf{A}_q[N_{-}^{w}] :=\mathbf{A}_q[N_{-}\cap X_w][[\mathcal{D}_{w}]^{-1}].
\end{align*}
This algebra $\mathbf{A}_q[N_{-}^{w}]$ is called \emph{a quantum unipotent cell}. A quantum unipotent cell has a $Q$-graded algebra structure in an obvious way. 
\end{definition}
The following map $\eta_{w,q}$ is called \emph{a quantum twist automorphism}. It is a quantum analogue of the (dual of) the twist automorphism $\eta_w^{\ast}\colon \mathbb{C}[N_-^w]\to \mathbb{C}[N_-^w]$, introduced by Berenstein, Fomin and Zelevinsky \cite{BFZ, BZ:Schu} (see Section \ref{s:introduction}). See \cite{KO2} for the precise argument of specialization at $q=1$.
\begin{proposition}[{\cite[Theorem 6.1]{KO2}}]\label{p:BZauto} Let $w\in W$. Then there exists a $\mathbb{Q}(q)$-algebra automorphism $\eta_{w,q}\colon\mathbf{A}_q[N_{-}^{w}]\to\mathbf{A}_q[N_{-}^{w}]$ given by 
\begin{align*}
[D_{u,u_{\lambda}}] & \mapsto q^{-(\lambda,\weight u-\lambda)}[D_{w\lambda,\lambda}]^{-1}[D_{u_{w\lambda},u}],&
[D_{w\lambda,\lambda}]^{-1} & \mapsto q^{(\lambda,w\lambda-\lambda)}[D_{w\lambda,\lambda}]
\end{align*}
for a weight vector $u\in V(\lambda)$ and $\lambda\in P_{+}$. In particular, $\weight\eta_{w,q}([x])=-\weight[x]$ for a homogeneous element $[x]\in\mathbf{A}_q[N_{-}^{w}]$.
\end{proposition}
It is easy to show that the inverse of the quantum twist automorphism is given by  
\begin{align}
\eta_{w,q}^{-1}([D_{u_{w\lambda}, u}])=q^{(\lambda, \weight u-w\lambda)}[D_{w\lambda, \lambda}]^{-1}[D_{u, u_{\lambda}}]\label{inverse}
\end{align}
for a weight vector $u\in V(\lambda)$ and $\lambda\in P_+$.
\section{Quantum Chamber Ansatz}\label{s:CA}
In this section, we prove quantum analogues of the Chamber Ansatz formulae for unipotent cells (Corollary \ref{c:CA}) by using the quantum twist automorphisms. A quantum analogue of the homomorphism $y_{\bm{i}}^{\ast}$ (see \eqref{clFeigin}) is known as the Feigin homomorphism. By the Feigin homomorphisms, we can realize quantum unipotent cells in quantum tori. Quantum Chamber Ansatz formulae provide explicit description of the variables of quantum tori in terms of elements of quantum unipotent cells. 
\begin{definition}\label{d:Feigin}
Let $\bm{i}=(i_1,\dots, i_{\ell})\in I^{\ell}$. \emph{The quantum affine space} (resp.~\emph{the quantum torus}) $\mathcal{P}_{\bm{i}}$ (resp.~$\mathcal{L}_{\bm{i}}$) is the unital associative $\mathbb{Q}(q)$-algebra generated by $t_1,\dots, t_{\ell}$ (resp.~$t_1^{\pm 1},\dots, t_{\ell}^{\pm 1}$) subject to the relations;
\begin{align*}
&t_jt_k=q^{(\alpha_{i_j}, \alpha_{i_k})}t_kt_j\;\text{for}\;1\leq j<k\leq \ell,\\
&t_kt_k^{-1}=t_k^{-1}t_k=1\;\text{for}\;1\leq k\leq \ell.
\end{align*} 
Define the $\mathbb{Q}(q)$-linear map $\Phi_{\bm{i}}\colon \mathbf{U}_q^-\to \mathcal{P}_{\bm{i}}$ by 
\[
x\mapsto  \sum_{\bm{a}=(a_1,\dots, a_{\ell})\in \mathbb{Z}_{\geq 0}^{\ell}}q_{\bm{i}}(\bm{a})(x, f_{i_1}^{(a_1)}\cdots f_{i_{\ell}}^{(a_{\ell})})_{L}t_1^{a_1}\cdots t_{\ell}^{a_{\ell}}
\]
where
\[
q_{\bm{i}}(\bm{a}):=\prod_{k=1}^{\ell}q_{i_k}^{a_k(a_k-1)/2}.
\]
Note that the all but finitely many summands in the right-hand side are zero. The map $\Phi_{\bm{i}}$ is called \emph{a Feigin homomorphism}.
\end{definition}
\begin{proposition}[{\cite[Theorem 2.1, Theorem 3.1]{Berenstein}}]\label{p:Feigin}
\textup{(1)} For $\bm{i}\in I^{\ell}$, the map $\Phi_{\bm{i}}$ is a $\mathbb{Q}(q)$-algebra homomorphism.

\textup{(2)} For $w\in W$ and $\bm{i}\in I(w)$, we have $\Ker \Phi_{\bm{i}}=\left(\mathbf{U}_{w}^{-}\right)^{\perp}$.

\textup{(3)} For $w\in W$, $\bm{i}=(i_1,\dots, i_{\ell})\in I(w)$ and $\lambda\in P_+$, we have 
\[
\Phi_{\bm{i}}\left(D_{w\lambda, \lambda}\right)=q_{\bm{i}}(\bm{a})t_1^{a_1}\cdots t_{\ell}^{a_{\ell}}
\]
where $\bm{a}=(a_1,\dots, a_{\ell})$ with $a_k:=\langle w_{\leq k}h_{i_{k}}, w\lambda\rangle$. Recall the notation in Definition \ref{d:Weyl}.
\end{proposition}
\begin{remark}
For any $\bm{i}=(i_1,\dots, i_{\ell})\in I^{\ell}$, we have $\Phi_{\bm{i}}((1-q_i^2)f_i)=\sum_{k; i_k=i}t_k$.
\end{remark}
\begin{definition}\label{d:Feiginembed}
Let $w\in W$ and $\bm{i}\in I(w)$. By Proposition \ref{p:Feigin} and the universality of localization, we have the embedding of the algebra $\mathbf{A}_q[N_-^{w}]\to \mathcal{L}_{\bm{i}}$, also denoted by $\Phi_{\bm{i}}$.
\end{definition}
The following is the main theorem in this paper. See also Corollary \ref{c:CA}. Recall the notation in Definition \ref{d:Weyl}. 
\begin{theorem}\label{t:monomial}
Let $w\in W$, $\bm{i}=(i_1,\dots, i_{\ell})\in I(w)$ and $k=1,\dots, \ell$. Then we have
\[
(\Phi_{\bm{i}}\circ \eta_{w,q}^{-1})([D_{w_{\leq k}\varpi_{i_{k}}, \varpi_{i_{k}}}])=\left(\prod_{j=1}^{k}q_{i_j}^{d_j(d_j+1)/2}\right) t_{1}^{-d_1}t_{2}^{-d_2}\cdots t_{k}^{-d_k},
\]
where $d_j:=\langle w_{\leq j}h_{i_j}, w_{\leq k}\varpi_{i_k}\rangle, j=1,\dots, k$.
\end{theorem}
\begin{remark}\label{r:BR}
Up to some conventions\footnote{The difference of conventions can be adjusted by regarding $q$ and $x_i$ in \cite{BR} as $q^{-1}$ and $(1-q_i^2)f_i$ in our paper respectively. Then the Feigin homomorphism $\underline{\Psi_{\bm{i}}}\colon \Bbbk_q[N^w]\to \mathcal{L}_{\bm{i}}$ ($\Bbbk=\mathbb{Q}(q^{\frac{1}{2}})$) in \cite[(6.17)]{BR} coincides with our $\Phi_{\bm{i}}\colon \mathbf{A}_q[N_-^{w}]\to \mathcal{L}_{\bm{i}}$ by extending our base field to $\mathbb{Q}(q^{\frac{1}{2}})$, and the generalized quantum minor $\Delta_{w\lambda}$ ($w\in W$, $\lambda\in P_+$) in \cite{BR} is equal to our $q^{-(w\lambda-\lambda, w\lambda-\lambda)/4-(w\lambda-\lambda, \rho)/2}D_{w\lambda, \lambda}$ ($\rho:=\sum_{i\in I}\varpi_i$). The (conjectural) quantum twist automorphism $\eta_{\bm{i}}$ in \cite{BR} corresponds to our $q^{-(-,\rho)}\circ \eta_{w,q}$ via $\Phi_{\bm{i}}$, here $q^{-(-,\rho)}$ is the algebra automorphism on $\mathbf{A}_q[N_-^{w}]$ given by $x\mapsto q^{-(\weight x,\rho)}x$. 
}, Theorem \ref{t:monomial} is a generalization of \cite[Corollary 1.2]{BR}, where they treat the case that $w$ is the square of an acyclic Coxeter element. Moreover, by Theorem \ref{t:monomial}, we can say that the quantum twist automorphism $\eta_{w, q}$ is a generalization of Berenstein-Rupel's quantum twist automorphism \cite[Theorem 2.10]{BR} and corresponds to the one proposed in \cite[Conjecture 2.12 (c), Conjecture 6.20]{BR} (see the footnote for the precise comparison). Note that Berenstein-Rupel's (conjectural) quantum twist automorphisms are proposed as automorphisms on certain upper quantum cluster algebras, but it is now shown that their upper quantum cluster algebra $\mathcal{U}_{\bm{i}}$ coincides with $\mathbf{A}_q[N_-^{w}]\otimes_{\mathbb{Q}(q)}\mathbb{Q}(q^{\frac{1}{2}})$ via $\Phi_{\bm{i}}$ by \cite[Theorem 8.2, Theorem 10.1]{GY:Memo} (see also the twist isomorphism in \cite{KO2} or Proposition \ref{p:BZisom}) and Theorem \ref{t:monomial}. This coincidence was also conjectured in \cite[Conjecture 2.12 (a)]{BR}. 

The compatibility between quantum twist automorphisms and dual canonical bases, which is proved in \cite[Theorem 6.1]{KO2}, corresponds to Berenstein-Rupel's conjectural property  \cite[Conjecture  2.17 (a)]{BR}. Note that our approach does not refer to quantum cluster algebra structures unlike Berenstein-Rupel's one.  
\end{remark}
\begin{proof}[{Proof of Theorem \ref{t:monomial}}]
If $w=e$, there is nothing to prove. From now on, we assume that $\ell(:=\ell (w))$ is greater than $0$. 

The proof is by induction on $k$. Let $k=1$. Take $\lambda\in P_+$ such that $\langle h_{i_1}, w\lambda \rangle <0$. Then it is easily seen that
\[
D_{s_{i_1}\varpi_{i_{1}}, \varpi_{i_{1}}}=[\langle h_{i_1}, w_{2\leq}\lambda\rangle]_{i_1}^{-1}D_{u_{w\lambda},\;e_{i_1}.u_{w\lambda}}.
\]
Hence, by \eqref{inverse},
\[
(\Phi_{\bm{i}}\circ \eta_{w, q}^{-1})([D_{s_{i_1}\varpi_{i_{1}}, \varpi_{i_{1}}}])=q_{i_1}^{\langle h_{i_1}, \lambda\rangle}[\langle h_{i_1}, w_{2\leq}\lambda\rangle]_{i_1}^{-1}\Phi_{\bm{i}}\left([D_{w\lambda, \lambda}]^{-1}[D_{e_{i_1}.u_{w\lambda}, u_{\lambda}}]\right).
\]
By Proposition \ref{p:Feigin} (3), we have
\begin{align*}
\Phi_{\bm{i}}\left([D_{w\lambda, \lambda}]^{-1}\right)&=q_{\bm{i}}(\bm{c})^{-1} t_{\ell}^{-c_{\ell}}\cdots t_1^{-c_1},\\
\Phi_{\bm{i}}\left([D_{e_{i_1}.u_{w\lambda}, u_{\lambda}}]\right)&=q_{\bm{i}}(\bm{c}-(1,0,\dots, 0))[c_1]_{i_1} t_1^{c_1-1}t_2^{c_2}\cdots t_{\ell}^{c_{\ell}},
\end{align*}
where $\bm{c}=(c_1,\dots, c_{\ell})$ with $c_j:=\langle h_{i_j}, w_{j+1\leq}\lambda \rangle$. Combining the above equalities, we obtain
\begin{align*}
(\Phi_{\bm{i}}\circ \eta_{w, q}^{-1})([D_{s_{i_1}\varpi_{i_{1}}, \varpi_{i_{1}}}])&=q_{i_1}^{\langle h_{i_1}, \lambda-w_{2\leq}\lambda-\sum_{j=2}^{\ell}c_j\alpha_{i_j}\rangle+1}t_1^{-1}\\
&=q_{i_1}t_1^{-1}.
\end{align*}
This proves the assertion in the case $k=1$. 

Assume that $k>1$. By Proposition \ref{p:replace} and Remark \ref{r:Demazure}, there exist $\lambda\in P_+$ and $b\in \mathscr{B}_w(\lambda)$ such that $D_{u_{w\lambda}, \Gup_{\lambda}(b)}=D_{w_{\leq k}\varpi_{i_{k}}, \varpi_{i_{k}}}$.
\begin{claim}\label{cl:minor}
$D_{u_{w\lambda}, \Gup_{\lambda}(\tilde{f}_{i_1}^{\mathrm{max}}b)}=D_{w_{\leq k}\varpi_{i_{k}}, s_{i_1}\varpi_{i_{k}}}$. Here $\tilde{f}_{i_1}^{\mathrm{max}}b:=\tilde{f}_{i_1}^{\varphi_{i_1}(b)}b=\tilde{f}_{i_1}^{\delta_{i_1, i_k}}b$. 
\end{claim}
\begin{proof}[Proof of Claim \ref{cl:minor}]
Let $\delta:=\delta_{i_1, i_k}$. Since $u_{s_{i_1}\varpi_{i_{k}}}=f_{i_1}^{\delta}.u_{\varpi_{i_{k}}}$, we have 
\[
D_{u_{w\lambda}, f_{i_1}^{(p)}.\Gup_{\lambda}(b)}=\begin{cases}D_{w_{\leq k}\varpi_{i_{k}}, s_{i_1}\varpi_{i_{k}}}\neq 0&\text{if}\; p=\delta,\\ 0 &\text{if}\;p>\delta.\end{cases}
\]
On the other hand, by Proposition \ref{p:EFaction}, 
\[
f_{i_1}^{(p)}.\Gup_{\lambda}(b)=\begin{cases}\Gup_{\lambda}(\tilde{f}_{i_1}^{\mathrm{max}}b)&\text{if}\; p=\varphi_{i_1}(b),\\ 0 &\text{if}\;p>\varphi_{i_1}(b),\end{cases}
\]
and, by Proposition \ref{p:Demazure},  $\tilde{f}_{i_1}^{\mathrm{max}}b\in \mathscr{B}_w(\lambda)$. Hence, 
\[
D_{u_{w\lambda}, f_{i_1}^{(p)}.\Gup_{\lambda}(b)}=\begin{cases}D_{u_{w\lambda}, \Gup_{\lambda}(\tilde{f}_{i_1}^{\mathrm{max}}b)}\neq 0&\text{if}\; p=\varphi_{i_1}(b),\\ 0 &\text{if}\;p>\varphi_{i_1}(b).\end{cases}
\]
Combining the above arguments, we obtain $\varphi_{i_1}(b)=\delta$ and $D_{w_{\leq k}\varpi_{i_{k}}, s_{i_1}\varpi_{i_{k}}}=D_{u_{w\lambda}, \Gup_{\lambda}(\tilde{f}_{i_1}^{\mathrm{max}}b)}$.%\qed
\end{proof}
We write $b_2:=\tilde{e}_{i_1}^{\mathrm{max}}b=\tilde{e}_{i_1}^{\mathrm{max}}\tilde{f}_{i_1}^{\mathrm{max}}b$. 
\begin{claim}\label{cl:Eaction}
We have
\[
D_{u_{w\lambda}, \Gup_{\lambda}(b_2)}=q_{i_1}^{(X-1-2\langle h_{i_1}, w_{\leq k}\varpi_{i_k}\rangle)X/2}D_{s_{i_1}\varpi_{i_1}, \varpi_{i_1}}^{X}D_{w_{\leq k}\varpi_{i_k}, \varpi_{i_k}},
\]
where $X:=-\langle h_{i_1}, w\lambda-w_{\leq k}\varpi_{i_k}\rangle$.
\end{claim}
\begin{proof}[Proof of Claim \ref{cl:Eaction}]
By \cite[Corollary 3.1.8]{Lus:intro}, for $p\in\mathbb{Z}_{\geq 0}$ and $x\in \mathbf{U}_q^-$, we have
\[
xe_{i_1}^{(p)}=\sum_{p'+p''+p'''=p}A(p', p'', p''')k_{i_1}^{-p'''}e_{i_1}^{(p'')}({_{i_1}e'})^{p'}(e'_{i_1})^{p'''}(x)k_{i_1}^{p'},
\]
where
\[
A(p', p'', p'''):=(-q_{i_1})^{p'''}q_{i_1}^{p'p''+p'p'''+p''p'''+{p'}^2}\frac{1}{(1-q_{i_1}^{2})^{p'}[p']_{i_1}!}\frac{1}{(1-q_{i_1}^{2})^{p'''}[p''']_{i_1}!}.
\]
Therefore, for $x\in \mathbf{U}_q^-$, we have 
\begin{align*}
&(D_{u_{w\lambda}, e_{i_1}^{(p)}.\Gup_{\lambda}(b)}, x)_L\\
&=(u_{w\lambda}, xe_{i_1}^{(p)}.\Gup_{\lambda}(b))_{\lambda}^{\varphi}\\
&=\sum_{p'+p''+p'''=p}A(p', p'', p''')(u_{w\lambda}, k_{i_1}^{-p'''}e_{i_1}^{(p'')}({_{i_1}e'})^{p'}(e'_{i_1})^{p'''}(x)k_{i_1}^{p'}.\Gup_{\lambda}(b))_{\lambda}^{\varphi}\\
&=\sum_{p'+p''=p}A(p', 0, p'')(u_{w\lambda}, k_{i_1}^{-p''}({_{i_1}e'})^{p'}(e'_{i_1})^{p''}(x)k_{i_1}^{p'}.\Gup_{\lambda}(b))_{\lambda}^{\varphi}\\
&=\sum_{p'+p''=p}A(p', 0, p'')q_{i_1}^{p'\langle h_{i_1}, w\lambda+\varpi_{i_k}-w_{\leq k}\varpi_{i_k}\rangle-p''\langle h_{i_1}, w\lambda\rangle}(D_{w_{\leq k}\varpi_{i_k}, \varpi_{i_k}}, ({_{i_1}e'})^{p'}(e'_{i_1})^{p''}(x))_{L}\\
&=\sum_{p'+p''=p}A(p', 0, p'')q_{i_1}^{p'\langle h_{i_1}, w\lambda-2w_{\leq k}\varpi_{i_k}\rangle-p''\langle h_{i_1}, w\lambda\rangle}(D_{s_{i_1}\varpi_{i_1}, \varpi_{i_1}}^{p}D_{w_{\leq k}\varpi_{i_k}, \varpi_{i_k}}, x)_{L}.
\end{align*}
Note that the last equality follows from \eqref{eq:adjoint} and Proposition \ref{p:comm}. Therefore we have 
\begin{align*}
&D_{u_{w\lambda}, e_{i_1}^{(p)}.\Gup_{\lambda}(b)}\\
&=\sum_{p'+p''=p}A(p', 0, p'')q_{i_1}^{p'\langle h_{i_1}, w\lambda-2w_{\leq k}\varpi_{i_k}\rangle-p''\langle h_{i_1}, w\lambda\rangle}D_{s_{i_1}\varpi_{i_1}, \varpi_{i_1}}^{p}D_{w_{\leq k}\varpi_{i_k}, \varpi_{i_k}}.
\end{align*}
In particular, since $\Gup_{\lambda}(b_2)=e_{i_1}^{(\varepsilon_{i_1}(b))}\Gup_{\lambda}(b)=e_{i_1}^{(-\langle h_{i_1}, w\lambda-w_{\leq k}\varpi_{i_k}\rangle)}\Gup_{\lambda}(b)$ by Claim \ref{cl:minor}, we have
\begin{align}
&D_{u_{w\lambda}, \Gup_{\lambda}(b_2)}\label{binomform}\\
&=\frac{q_{i_1}^{-\langle h_{i_1}, w_{\leq k}\varpi_{i_k}\rangle X}}{(1-q_{i_1}^2)^X}\left(\sum_{p'+p''=X}(-q_{i_1})^{p''}q_{i_1}^{p''X}\frac{1}{[p']_{i_1}![p'']_{i_1}!}\right) D_{s_{i_1}\varpi_{i_1}, \varpi_{i_1}}^{X}D_{w_{\leq k}\varpi_{i_k}, \varpi_{i_k}}. \notag
\end{align}
Recall that $X=-\langle h_{i_1}, w\lambda-w_{\leq k}\varpi_{i_k}\rangle$. By the way, the following equality is well-known. See for instance \cite[1.3.1]{Lus:intro}.
\[
\sum_{t=0}^{a}q^{t(a-1)}\frac{[a]!}{[t]![a-t]!}z^t=\prod_{j=0}^{a-1}(1+q^{2j}z)
\]
for $a\in \mathbb{Z}_{\geq 0}$. Substituting $q$ by $q_{i_1}$, $a$ by $X$ and $z$ by $-q_{i_1}^2$, we have
\[
\sum_{t=0}^{X}(-q_{i_1})^{t}q_{i_1}^{tX}\frac{[X]_{i_1}!}{[t]_{i_1}![X-t]_{i_1}!}=\prod_{j=1}^{X}(1-q_{i_1}^{2j}).
\]
Combining this equality with \eqref{binomform}, we obtain
\begin{align*}
D_{u_{w\lambda}, \Gup_{\lambda}(b_2)}&=\frac{q_{i_1}^{-\langle h_{i_1}, w_{\leq k}\varpi_{i_k}\rangle X}\prod_{j=1}^{X}(1-q_{i_1}^{2j})}{(1-q_{i_1}^2)^X[X]_{i_1}!}D_{s_{i_1}\varpi_{i_1}, \varpi_{i_1}}^{X}D_{w_{\leq k}\varpi_{i_k}, \varpi_{i_k}}\\
&=q_{i_1}^{(X-1-2\langle h_{i_1}, w_{\leq k}\varpi_{i_k}\rangle)X/2}D_{s_{i_1}\varpi_{i_1}, \varpi_{i_1}}^{X}D_{w_{\leq k}\varpi_{i_k}, \varpi_{i_k}}.
\end{align*}
%\qed
\end{proof}
By Claim \ref{cl:Eaction} and $(\Phi_{\bm{i}}\circ\eta_{w, q}^{-1})([D_{s_{i_1}\varpi_{i_1}, \varpi_{i_1}}])=q_{i_1}t_1^{-1}$, we have
\begin{align}
&(\Phi_{\bm{i}}\circ\eta_{w, q}^{-1})([D_{u_{w\lambda}, \Gup_{\lambda}(b_2)}])\label{calc'}\\
&=q_{i_1}^{(c_1+\langle s_{i_1}h_{i_1}, w_{\leq k}\varpi_{i_k}\rangle+1)X/2}t_1^{-X}(\Phi_{\bm{i}}\circ\eta_{w, q}^{-1})([D_{w_{\leq k}\varpi_{i_k}, \varpi_{i_k}}]).\notag
\end{align}
Since our aim is to calculate $(\Phi_{\bm{i}}\circ\eta_{w, q}^{-1})([D_{w_{\leq k}\varpi_{i_k}, \varpi_{i_k}}])$, we next describe $(\Phi_{\bm{i}}\circ\eta_{w, q}^{-1})([D_{u_{w\lambda}, \Gup_{\lambda}(b_2)}])$ in a different way. 
Now we have 
\begin{align}
\eta_{w, q}^{-1}([D_{u_{w\lambda}, \Gup_{\lambda}(b_2)}])&=q^{(\lambda, \weight b_2-w\lambda)}[D_{w\lambda, \lambda}]^{-1}[D_{\Gup_{\lambda}(b_2), u_{\lambda}}]\label{eq:invmin}\\
&=q^{(\lambda, \varpi_{i_k}-w_{\leq k}\varpi_{i_k}+X\alpha_{i_1})}[D_{w\lambda, \lambda}]^{-1}[D_{\Gup_{\lambda}(b_2), u_{\lambda}}].\notag
\end{align}
Moreover,
\begin{align}\label{Feigineq}
&\Phi_{\bm{i}}([D_{\Gup_{\lambda}(b_2), u_{\lambda}}])\\
&=\sum_{\bm{a}=(a_1,\dots, a_{\ell})\in\mathbb{Z}_{\geq 0}^{\ell}}q_{\bm{i}}(\bm{a})(\Gup_{\lambda}(b_2), f_{i_1}^{(a_1)}\cdots f_{i_{\ell}}^{(a_{\ell})}.u_{\lambda})_{\lambda}^{\varphi}t_1^{a_1}\cdots t_{\ell}^{a_{\ell}}\notag\\
&=\sum_{(a_2,\dots, a_{\ell})\in\mathbb{Z}_{\geq 0}^{\ell-1}}q_{\bm{i}}((0, a_2,\dots, a_{\ell}))(\Gup_{\lambda}(b_2), f_{i_2}^{(a_2)}\cdots f_{i_{\ell}}^{(a_{\ell})}.u_{\lambda})_{\lambda}^{\varphi}t_2^{a_2}\cdots t_{\ell}^{a_{\ell}}.\notag
\end{align}
The last equality holds because $e_{i_1}.\Gup_{\lambda}(b_2)=0$ by  Proposition \ref{p:EFaction}. 
Here we prepare one more claim. 
\begin{claim}\label{cl:induction}
Set $\mu_2:=w_{2\leq}\lambda$. Then $D_{u_{\mu_2}, \Gup_{\lambda}(b_2)}=D_{s_{i_1}w_{\leq k}\varpi_{i_k}, \varpi_{i_k}}$. 
\end{claim}
\begin{proof}[Proof of Claim \ref{cl:induction}]
By Propositions \ref{p:braid}, \ref{p:extremal}, \ref{p:kernel} and Claim \ref{cl:minor}, for $x\in \mathbf{U}_q^-$, we have
\begin{align*}
&(D_{s_{i_1}w_{\leq k}\varpi_{i_k}, \varpi_{i_k}}, x)_L\\
&=
\begin{cases}
(u_{s_{i_1}w_{\leq k}\varpi_{i_k}}, x.u_{\varpi_{i_k}})_{\varpi_{i_k}}^{\varphi}&\text{if}\;x\in \mathbf{U}_q^-\cap T_{i_1}(\mathbf{U}_q^-)=\Ker e'_{i_1},\\
0&\text{if}\;x\in f_{i_1}\mathbf{U}_q^-=(\Ker e'_{i_1})^{\perp},
\end{cases}\\
&=\begin{cases}
(u_{w_{\leq k}\varpi_{i_k}}, T_{i_1}^{-1}(x).u_{s_{i_1}\varpi_{i_k}})_{\varpi_{i_k}}^{\varphi}&\text{if}\;x\in \mathbf{U}_q^-\cap T_{i_1}(\mathbf{U}_q^-),\\
0&\text{if}\;x\in f_{i_1}\mathbf{U}_q^-=(\Ker e'_{i_1})^{\perp},
\end{cases}\\
&=\begin{cases}
(u_{w\lambda}, T_{i_1}^{-1}(x).\Gup_{\lambda}(\tilde{f}_{i_1}^{\mathrm{max}}b))_{\lambda}^{\varphi}&\text{if}\;x\in \mathbf{U}_q^-\cap T_{i_1}(\mathbf{U}_q^-),\\
0&\text{if}\;x\in f_{i_1}\mathbf{U}_q^-=(\Ker e'_{i_1})^{\perp},
\end{cases}\\
&=(u_{\mu_2}, x.\Gup_{\lambda}(b_2))_{\lambda}^{\varphi}=(D_{u_{\mu_2}, \Gup_{\lambda}(b_2)}, x)_L.
\end{align*}
This completes the proof.%\qed
\end{proof}
Set $\bm{i}_{2\leq}:=(i_2,\dots, i_{\ell})$ and identify $\mathcal{L}_{\bm{i}_{2\leq}}$ with the subalgebra of $\mathcal{L}_{\bm{i}}$ generated by $t_{2}^{\pm 1},\dots, t_{\ell}^{\pm 1}$. Write 
\begin{align*}
C_2&:=\prod_{j=2}^{k}q_{i_j}^{\langle s_{i_1}w_{\leq j}h_{i_j}, s_{i_1}w_{\leq k}\varpi_{i_k}\rangle(\langle s_{i_1}w_{\leq j}h_{i_j}, s_{i_1}w_{\leq k}\varpi_{i_k}\rangle+1)/2}=\prod_{j=2}^{k}q_{i_j}^{d_j(d_j+1)/2}.
\end{align*}
By our induction assumption, Proposition \ref{p:Feigin} (3) and Claim \ref{cl:induction}, we have
\begin{align*}
&
C_2
t_2^{-\langle w_{\leq 2}h_{i_2}, w_{\leq k}\varpi_{i_k}\rangle}\cdots t_{k}^{-\langle w_{\leq k}h_{i_{k}}, w_{\leq k}\varpi_{i_k}\rangle}
\left(=:C_2\dprod_{j=2,\dots, k} t_j^{-\langle w_{\leq j}h_{i_j}, w_{\leq k}\varpi_{i_k}\rangle}\right)
\\ 
&=(\Phi_{\bm{i}_{2\leq}}\circ \eta_{w_{2\leq}, q}^{-1})([D_{s_{i_1}w_{\leq k}\varpi_{i_k}, \varpi_{i_k}}])\\ 
&
=(\Phi_{\bm{i}_{2\leq}}\circ \eta_{w_{2\leq}, q}^{-1})([D_{u_{\mu_2}, \Gup_{\lambda}(b_2)}])
\\
&=\Phi_{\bm{i}_{2\leq}}(q^{(\lambda, \varpi_{i_k}-w_{\leq k}\varpi_{i_k}+\langle h_{i_1}, w_{\leq k}\varpi_{i_k}\rangle\alpha_{i_1})}[D_{w_{2\leq}\lambda, \lambda}]^{-1}[D_{\Gup_{\lambda}(b_2), u_{\lambda}}])\\ 
&=q^{(\lambda, \varpi_{i_k}-w_{\leq k}\varpi_{i_k}+\langle h_{i_1}, w_{\leq k}\varpi_{i_k}\rangle\alpha_{i_1})}
q_{\bm{i}_{2\leq}}(\bm{c}')^{-1}
t_{\ell}^{-c_{\ell}}\cdots t_{2}^{-c_{2}}\\
&\hspace{11pt}\times\sum_{\bm{a}'=(a_2,\dots, a_{\ell})\in \mathbb{Z}_{\geq 0}^{\ell-1}}q_{\bm{i}_{2\leq}}(\bm{a}')(\Gup_{\lambda}(b_2), f_{i_2}^{(a_2)}\cdots f_{i_{\ell}}^{(a_{\ell})}.u_{\lambda})_{\lambda}^{\varphi}t_2^{a_2}\cdots t_{\ell}^{a_{\ell}},
\end{align*}
where $\bm{c}'=(c_2,\dots, c_{\ell})$ with $c_j:=\langle h_{i_j}, w_{j+1\leq} \lambda\rangle$. Therefore, 
\begin{align}
&\sum_{\bm{a}'=(a_2,\dots, a_{\ell})\in \mathbb{Z}_{\geq 0}^{\ell-1}}q_{\bm{i}_{2\leq}}(\bm{a}')(\Gup_{\lambda}(b_2), f_{i_2}^{(a_2)}\cdots f_{i_{\ell}}^{(a_{\ell})}.u_{\lambda})_{\lambda}^{\varphi}t_2^{a_2}\cdots t_{\ell}^{a_{\ell}}\label{Feigineq'}\\
&=
C_2q^{-(\lambda, \varpi_{i_k}-w_{\leq k}\varpi_{i_k}+\langle h_{i_1}, w_{\leq k}\varpi_{i_k}\rangle\alpha_{i_1})}
q_{\bm{i}_{2\leq}}(\bm{c}')
t_{2}^{c_{2}}\cdots t_{\ell}^{c_{\ell}}
\dprod_{j=2,\dots, k}
t_j^{-\langle w_{\leq j}h_{i_j}, w_{\leq k}\varpi_{i_k}\rangle}.\notag
\end{align}
Combining \eqref{eq:invmin}, \eqref{Feigineq} and \eqref{Feigineq'}, we obtain the following equality ($\bm{c}=(c_1,\dotsm c_{\ell}), c_1:=\langle h_{i_1}, w_{2\leq} \lambda\rangle$):
\begin{align}
&(\Phi_{\bm{i}}\circ\eta_{w, q}^{-1})([D_{u_{w\lambda}, \Gup_{\lambda}(b_2)}])\label{calc}\\
&=q^{(\lambda, \varpi_{i_k}-w_{\leq k}\varpi_{i_k}+X\alpha_{i_1})}\Phi_{\bm{i}}([D_{w\lambda, \lambda}]^{-1}[D_{\Gup_{\lambda}(b_2), u_{\lambda}}])\notag\\ 
&=C_2q_{i_1}^{-\langle h_{i_1}, \lambda\rangle\langle h_{i_1}, w\lambda\rangle}q_{\bm{i}}(\bm{c})^{-1}q_{\bm{i}_{2\leq}}(\bm{c}')t_{\ell}^{-c_{\ell}}\dots t_{1}^{-c_{1}}t_{2}^{c_{2}}\cdots t_{\ell}^{c_{\ell}}\dprod_{j=2,\dots, k}t_j^{-\langle w_{\leq j}h_{i_j}, w_{\leq k}\varpi_{i_k}\rangle}\notag\\ 
&=C_2q_{i_1}^{c_1(c_1+1)/2}t_{1}^{-c_{1}}\dprod_{j=2,\dots, k}t_j^{-\langle w_{\leq j}h_{i_j}, w_{\leq k}\varpi_{i_k}\rangle}.\notag
\end{align}
Recall that $X=-\langle h_{i_1}, w\lambda-w_{\leq k}\varpi_{i_k}\rangle=c_1-\langle s_{i_1}h_{i_1}, w_{\leq k}\varpi_{i_k}\rangle$. By \eqref{calc} and \eqref{calc'}, we obtain
\begin{align*}
&(\Phi_{\bm{i}}\circ\eta_{w, q}^{-1})([D_{w_{\leq k}\varpi_{i_k}, \varpi_{i_k}}])\\
&=C_2q_{i_1}^{-(c_1+\langle s_{i_1}h_{i_1}, w_{\leq k}\varpi_{i_k}\rangle+1)X/2+c_1(c_1+1)/2}t_{1}^{-\langle s_{i_1}h_{i_1}, w_{\leq k}\varpi_{i_k}\rangle}\dprod_{j=2,\dots, k}t_j^{-\langle w_{\leq j}h_{i_j}, w_{\leq k}\varpi_{i_k}\rangle}\\
&=C_2q_{i_1}^{\langle s_{i_1}h_{i_1}, w_{\leq k}\varpi_{i_k}\rangle(\langle s_{i_1}h_{i_1}, w_{\leq k}\varpi_{i_k}\rangle+1)/2}t_{1}^{-\langle s_{i_1}h_{i_1}, w_{\leq k}\varpi_{i_k}\rangle}\dprod_{j=2,\dots, k}t_j^{-\langle w_{\leq j}h_{i_j}, w_{\leq k}\varpi_{i_k}\rangle}.
\end{align*}
This completes the proof. %\qed
\end{proof}
The following is a direct corollary of Theorem \ref{t:monomial}. These equalities are exact quantum analogues of the Chamber Ansatz formulae for unipotent cells \cite[Theorem 1.4]{BFZ}, \cite[Theorem 1.4]{BZ:Schu}. 
\begin{corollary}\label{c:CA}
Let $w\in W$ and $\bm{i}=(i_1,\dots, i_{\ell})\in I(w)$. For $j=1,\dots, \ell$, set 
\[
D_{w_{\leq j}\varpi_{i_{j}}, \varpi_{i_{j}}}^{\prime\; (\bm{i})}:=(\Phi_{\bm{i}}\circ \eta_{w, q}^{-1})([D_{w_{\leq j}\varpi_{i_{j}}, \varpi_{i_{j}}}]).
\]
By Theorem \ref{t:monomial}, these elements are Laurent monomials in $\mathcal{L}_{\bm{i}}$. Then, for $k=1,\dots, \ell$, 
\begin{align}
t_k=q^{N_k}(D_{w_{\leq k-1}\varpi_{i_{k}}, \varpi_{i_{k}}}^{\prime\; (\bm{i})})^{-1}(D_{w_{\leq k}\varpi_{i_{k}}, \varpi_{i_{k}}}^{\prime\; (\bm{i})})^{-1}\dprod_{j\in I\setminus\{i_k\}}(D_{w_{\leq k}\varpi_{j}, \varpi_{j}}^{\prime\; (\bm{i})})^{-a_{j, i_k}}\label{eq:chamber}
\end{align}
for some $N_k\in \mathbb{Z}$, here $\overrightarrow{\prod}$ stands for the ordered multiplication according to an arbitrarily fixed ordering on $I\setminus\{i_k\}$. More precisely, $N_k$ is given by 
\begin{align*}
N_k=&\frac{(\alpha_{i_k}, \alpha_{i_k})}{2}+(\varpi_{i_{k}}, w_{\leq k-1}\varpi_{i_{k}}-\varpi_{i_{k}})-\sum_{j\in I\setminus\{i_k\}}\left(\begin{array}{c}
-a_{j, i_k}\\
2
\end{array}\right)(\varpi_j, w_{\leq k}\varpi_j-\varpi_j)\\
&-\sum_{j, j'\in I\setminus\{i_k\};j<j'}a_{j, i_k}a_{j', i_k}(\varpi_j, w_{\leq k}\varpi_{j'}-\varpi_{j'}).
\end{align*}
\end{corollary}
\begin{proof}
The validity of the equality \eqref{eq:chamber} for some integer $N_k$ can be proved in exactly the same way as \cite[Theorem 4.3]{BZ:Schu} by Theorem \ref{t:monomial} and the relations among $t_j$'s. Henceforth we calculate $N_k$ explicitly. We prepare the operations called \emph{the dual bar-involutions}. There exists a $\mathbb{Q}$-linear automorphism $\sigma$ on $\mathbf{A}_q[N_-^{w}]$ characterized by  
\begin{align*}
\sigma(q)=q^{-1}&&&\sigma([f_i])=-q_i^2[f_i]&&&\sigma(xy)=q^{(\weight x, \weight y)}\sigma(y)\sigma(x)
\end{align*}
for $i\in I$ and homogeneous elements $x, y\in \mathbf{A}_q[N_-^{w}]$ (see \cite[Proposition 4.9]{KO2}). Note that $\sigma((1-q_i^2)f_i)=(1-q_i^2)f_i$. In fact, every element of the projected dual canonical basis $[\mathbf{B}^{\mathrm{up}}]$ is fixed by $\sigma$. Moreover, we have $\eta_{w, q}\circ \sigma=\sigma \circ \eta_{w, q}$ \cite[Theorem 6.1]{KO2}.

On the other hand, define $\sigma_{\bm{i}}$ as a $\mathbb{Q}$-linear automorphism on $\mathcal{L}_{\bm{i}}$ given by $f(q)\cdot q_{\bm{i}}(\bm{a})t_1^{a_1}\cdots t_{\ell}^{a_{\ell}}\mapsto f(q^{-1})\cdot q_{\bm{i}}(\bm{a})t_1^{a_1}\cdots t_{\ell}^{a_{\ell}}$ for $\bm{a}=(a_1,\dots, a_{\ell})\in \mathbb{Z}^{\ell}$ and $f(q)\in \mathbb{Q}(q)$ (see Definition \ref{d:Feigin} for the definition of $q_{\bm{i}}(\bm{a})$, which is obviously extended to $ \mathbb{Z}^{\ell}$). Then $\sigma_{\bm{i}}$ satisfies 
\begin{align*}
\sigma_{\bm{i}}(q)=q^{-1}&&&\sigma_{\bm{i}}(t_j)=t_j&&&\sigma_{\bm{i}}(xy)=q^{(\weight x, \weight y)}\sigma_{\bm{i}}(y)\sigma_{\bm{i}}(x)
\end{align*}
for $j=1,\dots, \ell$ and homogeneous elements $x, y\in \mathcal{L}_{\bm{i}}$, where $\mathcal{L}_{\bm{i}}$ is regarded as a $Q$-graded algebra by $\weight t_j=-\alpha_{i_j}$, $j=1,\dots, \ell$. Remark that this involution $\sigma_{\bm{i}}$ is slightly different from the bar-involution in \cite{BR}. We can easily check $\sigma_{\bm{i}}\circ \Phi_{\bm{i}}=\Phi_{\bm{i}}\circ \sigma$.

Let us return to the calculation of $N_k$. It follows from Theorem \ref{t:monomial} and the equality \eqref{eq:chamber} for some integer $N_k$ that there exist $n_1, n_2\in \mathbb{Z}$ and $\bm{c}=(c_1,\dots, c_{k-1}, 0,\dots, 0)\in \mathbb{Z}^{\ell}$ such that 
\begin{align*}
&D_{w_{\leq k}\varpi_{i_{k}}, \varpi_{i_{k}}}^{\prime\; (\bm{i})}D_{w_{\leq k-1}\varpi_{i_{k}}, \varpi_{i_{k}}}^{\prime\; (\bm{i})}=q^{n_1}\cdot q_{i_k}q_{\bm{i}}(\bm{c})t_1^{c_1}\cdots t_{k-1}^{c_{k-1}}t_k^{-1},\\
&\dprod_{j\in I\setminus\{i_k\}}(D_{w_{\leq k}\varpi_{j}, \varpi_{j}}^{\prime\; (\bm{i})})^{-a_{j, i_k}}=q^{n_2}\cdot q_{\bm{i}}(\bm{c})t_1^{c_1}\cdots t_{k-1}^{c_{k-1}}. 
\end{align*}
Then $N_k=\frac{(\alpha_{i_k}, \alpha_{i_k})}{2}+n_1-n_2$. Remark that the first (resp.~second) equality above implies that $q^{-n_1}$ (resp.~$q^{-n_2}$) times the left-hand side is $\sigma_{\bm{i}}$-invariant. By the way, since $\sigma_{\bm{i}}\circ (\Phi_{\bm{i}}\circ \eta_{w, q}^{-1})=(\Phi_{\bm{i}}\circ \eta_{w, q}^{-1})\circ \sigma$, the integers $n_1, n_2$ are unique integers such that $q^{-n_1}[D_{w_{\leq k}\varpi_{i_{k}}, \varpi_{i_{k}}}][D_{w_{\leq k-1}\varpi_{i_{k}}, \varpi_{i_{k}}}]$ and $q^{-n_2}\overrightarrow{\prod}_{j\in I\setminus\{i_k\}}[D_{w_{\leq k}\varpi_{j}, \varpi_{j}}]^{-a_{j, i_k}}$ are $\sigma$-invariant. Then, by using Proposition \ref{p:comm}, we can directly check that 
\begin{align*}
n_1=&(\varpi_{i_{k}}, w_{\leq k-1}\varpi_{i_{k}}-\varpi_{i_{k}}),\\
n_2=&\sum_{j\in I\setminus\{i_k\}}\left(\begin{array}{c}
-a_{j, i_k}\\
2
\end{array}\right)(\varpi_j, w_{\leq k}\varpi_j-\varpi_j)\\
&+\sum_{j, j'\in I\setminus\{i_k\};j<j'}a_{j, i_k}a_{j', i_k}(\varpi_j, w_{\leq k}\varpi_{j'}-\varpi_{j'}),
\end{align*}
which completes the calculation of $N_k$. %\qed
\end{proof}
\appendix
\section{Comparison with the Cauchon generators}\label{s:Cauchon}
In this appendix, we clarify an explicit relation between our quantum Chamber Ansatz formulae and the description of Cauchon generators given by Geiger-Yakimov \cite{GeY} and Lenagan-Yakimov \cite{LeY}. 

We review the results in \cite{GeY, LeY} briefly. For $w\in W$ and $\bm{i}=(i_1,\dots, i_{\ell})\in I(w)$, set 
\begin{align*}
\mathbf{A}_q[N_{-}(w)]&:=\spn_{\mathbb{Q}(q)}\left\{(T_{i_{1}}^{-1}\cdots T_{i_{\ell-1}}^{-1})(f_{i_{\ell}}^{c_{\ell}})\cdots T_{i_{1}}^{-1}(f_{i_{2}}^{c_{2}})f_{i_{1}}^{c_{1}}
\mid (c_1,\dots, c_{\ell})\in \mathbb{Z}_{\geq 0}^{\ell}\right\},\\
&=\spn_{\mathbb{Q}(q)}\left\{ D_{u_{w\lambda}, u}\mid u\in V(\lambda), \lambda\in P_{+}\right\}. 
\end{align*}
Then it is known that $\mathbf{A}_q[N_{-}(w)]$ is a $\mathbb{Q}(q)$-subalgebra of $\mathbf{U}_q^-$, and the first definition does not depend on the choice of $\bm{i}\in I(w)$ \cite[subsection 2.2]{DKP}, \cite[Proposition 40.2.1]{Lus:intro}. See, for example,  \cite[Propositions 5.14, 5.17]{KO2} for the second presentation. This subalgebra $\mathbf{A}_q[N_{-}(w)]$ is called \emph{a quantum unipotent subgroup}. 

Recall the set $\mathcal{D}_{w}$ in Proposition \ref{p:comm}. Then, by \cite[Corollary 6.18]{Kimura:qunip}, $\mathcal{D}_{w}$ is an Ore set of $\mathbf{A}_q[N_{-}(w)]$, hence we can define the localization 
\[
\mathbf{A}_q[N_{-}(w)\cap\dot{w}G_{0}]:=\mathbf{A}_q[N_{-}(w)][\mathcal{D}_{w}^{-1}].
\]
See \cite{KO2} for the meaning of this notation. In \cite{LeY}, this algebra (modulo some difference of conventions) is considered as a quantum analogue of the coordinate ring of $N_-^w$ ($N_-^w$ is isomorphic to $R_{e, w}$ in \cite{LeY}). In fact, the algebra $\mathbf{A}_q[N_{-}(w)\cap\dot{w}G_{0}]$ is isomorphic to $\mathbf{A}_q[N_{-}^w]$, but this isomorphism, called \emph{the twist isomorphism} in \cite{KO2}, is given in a non-trivial way:
\begin{proposition}[{\cite[Theorem 5.19]{KO2}}]\label{p:BZisom} There exists a $\mathbb{Q}(q)$-algebra isomorphism $\gamma_{w,q}\colon\mathbf{A}_q[N_{-}^{w}]\to\mathbf{A}_q[N_{-}(w)\cap\dot{w}G_{0}]$ given by 
\begin{align*}
[D_{u,u_{\lambda}}] & \mapsto q^{-(\lambda,\weight u-\lambda)}D_{w\lambda,\lambda}^{-1}D_{u_{w\lambda},u},&
[D_{w\lambda,\lambda}]^{-1} & \mapsto q^{(\lambda,w\lambda-\lambda)}D_{w\lambda,\lambda}
\end{align*}
for a weight vector $u\in V(\lambda)$ and $\lambda\in P_{+}$. 
\end{proposition}
\begin{remark}\label{r:BZisom}
The twist isomorphism $\gamma_{w,q}$ is of the same form as the twist map $\eta_{w,q}$. In particular, $\gamma_{w,q}^{-1}(x)=\eta_{w,q}^{-1}([x])$ for $x\in \mathbf{A}_q[N_{-}(w)]$. See \cite{KO2} (cf. also \cite[subsection 3.8, section 4]{Yak:inv}) for more details of $\gamma_{w,q}$.   
\end{remark}
For $w\in W$ and $\bm{i}=(i_1,\dots, i_{\ell})\in I(w)$, set 
\[
F_{-1}(\beta_{k}, \bm{i}):=(1-q_{i_k}^2)(T_{i_{1}}^{-1}\cdots T_{i_{k-1}}^{-1})(f_{i_{k}})
\]
for $k=1,\dots, \ell$, where $\beta_{k}:=w_{\leq k-1}\alpha_{i_k}(=-\weight F_{-1}(\beta_{k}, \bm{i}))$. Then, in fact, $F_{-1}(\beta_{k}, \bm{i})$ is an element of dual canonical basis \cite[Proposition 4.26]{Kimura:qunip}. By the Levendorskii-Soibelman straightening law (see, for example, \cite[Theorem 4.27]{Kimura:qunip}), the quantum unipotent subgroup $\mathbf{A}_q[N_{-}(w)]$ is presented as an iterated Ore extension: 
\begin{align}
\mathbf{A}_q[N_{-}(w)]=\mathbb{Q}(q)[F_{-1}(\beta_{\ell}, \bm{i})][F_{-1}(\beta_{\ell-1}, \bm{i}); \sigma_{\ell-1}, \delta_{\ell-1}]\cdots [F_{-1}(\beta_{1}, \bm{i}); \sigma_{1}, \delta_{1}],  \label{eq:CGL}
\end{align}
here $\sigma_{k}$ is an automorphism and $\delta_k$ is a left $\sigma_k$-skew derivation of the subalgebra 
$\mathbb{Q}(q)[F_{-1}(\beta_{\ell}, \bm{i})]$
$[F_{-1}(\beta_{\ell-1}, \bm{i}); \sigma_{\ell-1}, \delta_{\ell-1}]\cdots [F_{-1}(\beta_{k+1}, \bm{i}); \sigma_{k+1}, \delta_{k+1}]$ for $k=1,\dots, \ell-1$. Moreover, the presentation \eqref{eq:CGL} gives rise to a (torsion-free symmetric) Cauchon-Goodearl-Letzter (CGL) extension. See, for example, \cite[subsection 2.4]{GeY}, \cite[subsection 2.2]{LeY} and references therein for the precise definitions of each terminologies. 

Cauchon's method of deleting derivations \cite[section 3]{Cau}, which is applicable to all CGL extension, provides non-zero elements $y_{k}, k=1,\dots, \ell$ of the skew field $\fraction (\mathbf{A}_q[N_{-}(w)])$ of fractions of $\mathbf{A}_q[N_{-}(w)]$ associated to the presentation \eqref{eq:CGL}, which satisfy the following properties: 
\begin{itemize}
\item[(1)] the $\mathbb{Q}(q)$-subalgebra of $\fraction (\mathbf{A}_q[N_{-}(w)])$ generated by $y_{k}^{\pm 1}, k=1,\dots, \ell$ is isomorphic to the quantum torus subject to the relations; 
\begin{align*}
y_jy_k=q^{-(\beta_j, \beta_k)}y_ky_j
\end{align*}
for $1\leq j< k\leq \ell$. We write this quantum torus as $\mathcal{Y}_{\bm{i}}$. 
\item[(2)] $\mathbf{A}_q[N_{-}(w)]\subset \mathcal{Y}_{\bm{i}}$ (Note that $y_{k}\notin\mathbf{A}_q[N_{-}(w)]$ in general). 
\end{itemize}
The elements $y_{k}, k=1,\dots, \ell$ are called \emph{Cauchon generators}. See \cite[subsection 2.4]{GeY} for a concise summary of their precise construction. Here we should remark that Cauchon's method produces the elements $y_{k}, k=1,\dots, \ell$ in order, and we denote the $k$-th element by $y_k$, that is, our $(y_{\ell},\dots, y_1)$ corresponds to $(x_1^{(2)}, \dots, x_{\ell}^{(2)})$ in \cite[subsection 2.4]{GeY}. Note that our $(F_{-1}(\beta_{\ell}, \bm{i}),\dots, F_{-1}(\beta_{1}, \bm{i}))$ corresponds to $(x_1^{(\ell+1)}, \dots, x_{\ell}^{(\ell+1)})$. 

\begin{remark}
The Cauchon generators $y_{k}, k=1,\dots, \ell$ are not determined only from the algebra structure of $\mathbf{A}_q[N_{-}(w)]$, but depend on the presentation of $\mathbf{A}_q[N_{-}(w)]$ as a CGL-extension. 
\end{remark}
 We introduce one more convenient notation. For $\bm{a}=(a_1,\dots, a_{\ell})\in \mathbb{Z}^{\ell}$, set 
\[
y^{\bm{a}}:=q^{\sum_{k=1}^{\ell}a_k(a_k-1)(\beta_k, \beta_k)/4+\sum_{k>k'}a_ka_{k'}(\beta_k, \beta_{k'})}y_1^{a_1}\cdots y_{\ell}^{a_{\ell}}. 
\]

Geiger-Yakimov and Lenagan-Yakimov gave a simple (but highly non-trivial)  explicit description of $y_{k}, k=1,\dots, \ell$ by using unipotent quantum minors as follows.
\begin{proposition}[{\cite[Theorem 3.2]{GeY}, \cite[Theorem 8.1]{LeY}}]\label{p:Cauchon}
For $k=1,\dots, \ell$, we have 
\[
D_{w_{\leq k}\varpi_{i_{k}}, \varpi_{i_{k}}}=y^{\bm{c}_{\leq k}}
\]
in $\fraction (\mathbf{A}_q[N_{-}(w)])$, where $\bm{c}_{\leq k}=(c_1^{(k)},\dots, c_{\ell}^{(k)})$ is given by 
\[
c_{j}^{(k)}:=\begin{cases}
1&\text{if }i_j=i_k\ \text{and }j\leq k,\\
0&\text{otherwise} 
\end{cases}
\]
for $j=1,\dots, \ell$. 
\end{proposition}
\begin{remark}
Proposition \ref{p:Cauchon} is the direct translation of the statements in \cite{GeY, LeY} by our conventions (The statement \cite[Theorem 8.1]{LeY} is more general). Here we explain one way of translation from the conventions in \cite{GeY} to ours: 

We identify $\mathcal{U}^-$ in \cite{GeY} with our $\mathbf{U}_q^{-}$ in the obvious way. Then $b_{y, w}^{\lambda}$, $y, w\in W$ and $\lambda\in P_+$ in \cite{GeY} is equal to our $q^{-(w\lambda-y\lambda, w\lambda-y\lambda)/2}\ast (D_{w\lambda, y\lambda})$. Moreover, $\mathcal{U}^-[w^{-1}]$ in \cite{GeY} is isomorphic to our $\mathbf{A}_q[N_{-}(w)]$ via 
$\ast\circ \Theta_{w}$, here $\Theta_{w}$ is the map \cite[(6.1)]{LeY}, which they call the quantum twist map. By \cite[Proposition 3.3, Proposition 3.16, Theorem 3.22]{KO1} and the observation above, the isomorphism $\ast\circ \Theta_{w}$ gives the following correspondence:  
\begin{align*}
(T_{i_{\ell}}\cdots T_{i_{k+1}})(f_{i_{k}})&\mapsto (1-q_{i_k}^2)^{-1}F_{-1}(\beta_{k}, \bm{i}),\\
\Delta_{\bm{i}^{\mathrm{rev}}, \ell-k+1}:=b_{s_{i_{\ell}}\cdots s_{i_{k+1}}, w^{-1}}^{\varpi_{i_k}}&\mapsto q^{-(w_{\leq k}\varpi_{i_k}-\varpi_{i_k}, w_{\leq k}\varpi_{i_k}-\varpi_{i_k})/2}D_{w_{\leq k}\varpi_{i_k}, \varpi_{i_k}},
\end{align*}
here $\bm{i}^{\mathrm{rev}}:=(i_{\ell},\dots, i_1)$. In particular, the presentation of  $\mathcal{U}^-[w^{-1}]$ in \cite[(2.16)]{GeY} associated to $\bm{i}^{\mathrm{rev}}$ is transferred to the presentation \eqref{eq:CGL} of $\mathbf{A}_q[N_{-}(w)]$ modulo scalar multiple of generators $F_{-1}(\beta_{k}, \bm{i})$. Therefore, taking this scalar multiple and the relabeling of Cauchon generators into account, we can deduce Proposition \ref{p:Cauchon} from \cite[Theorem 3.2]{GeY} associated with $\mathcal{U}^-[w^{-1}]$ and $\bm{i}^{\mathrm{rev}}$.
\end{remark}
By Proposition \ref{p:Cauchon}, the inclusion $\mathbf{A}_q[N_{-}(w)]\subset \mathcal{Y}_{\bm{i}}$ extends to $\mathbf{A}_q[N_{-}(w)\cap\dot{w}G_{0}]\subset \mathcal{Y}_{\bm{i}}$. Hence we have an injective algebra homomorphism 
\[
\Psi_{\bm{i}}\colon \mathbf{A}_q[N_{-}^{w}]\xrightarrow[]{\gamma_{w,q}}\mathbf{A}_q[N_{-}(w)\cap\dot{w}G_{0}]\hookrightarrow \mathcal{Y}_{\bm{i}}. 
\]

Now we have two kinds of embedding of $\mathbf{A}_q[N_{-}^{w}]$ into quantum tori: 
\[
\mathcal{L}_{\bm{i}}\xleftarrow[]{\Phi_{\bm{i}}} \mathbf{A}_q[N_{-}^{w}]\xrightarrow[]{\Psi_{\bm{i}}} \mathcal{Y}_{\bm{i}}. 
\]

We conclude this appendix by clarifying an explicit relation between these embeddings. 
\begin{theorem}\label{t:comparison}
Let $w\in W$ and $\bm{i}=(i_1,\dots, i_{\ell})\in I(w)$. For $k=1,\dots, \ell$, define $\bm{a}_{\leq k}:=(a_1^{(k)},\dots, a_{\ell}^{(k)})$ by 
\[
a_{j}^{(k)}:=\begin{cases}
-a_{i_j, i_k}&\text{for } j<k,\\
-1&\text{for } j=k,\\
0&\text{for } j>k. 
\end{cases}
\]
Then the assignment
\[
t_k\mapsto y^{\bm{a}_{\leq k}}
\]
defines a $\mathbb{Q}(q)$-algebra isomorphism $M_{\bm{i}}\colon \mathcal{L}_{\bm{i}}\to \mathcal{Y}_{\bm{i}}$ such that $M_{\bm{i}}\circ \Phi_{\bm{i}}=\Psi_{\bm{i}}$. 
\end{theorem}
\begin{proof}
Since $\mathbf{A}_q[N_-^w]$ is regarded as a subalgebra of its skew field $\fraction (\mathbf{A}_q[N_-^w])$ of fractions (cf.~Proposition \ref{p:BZisom}), we can consider a $\mathbb{Q}(q)$-subalgebra $\mathcal{A}_{\bm{i}}$ of $\fraction (\mathbf{A}_q[N_-^w])$ generated by $\mathbf{A}_q[N_-^w]$ and $(\eta_{w,q}^{-1}([D_{w_{\leq k}\varpi_{i_{k}}, \varpi_{i_{k}}}]))^{-1}$, $k=1,\dots, \ell$. Then, by Theorem \ref{t:monomial}, $\Phi_{\bm{i}}$ extends to an isomorphism from $\mathcal{A}_{\bm{i}}$ to $\mathcal{L}_{\bm{i}}$, denoted again by $\Phi_{\bm{i}}$. On the other hand, by Remark \ref{r:BZisom} and Proposition \ref{p:Cauchon}, we have 
\begin{align}
\Psi_{\bm{i}}(\eta_{w,q}^{-1}([D_{w_{\leq k}\varpi_{i_{k}}, \varpi_{i_{k}}}]))=\Psi_{\bm{i}}(\gamma_{w, q}^{-1}(D_{w_{\leq k}\varpi_{i_{k}}, \varpi_{i_{k}}}))=y^{\bm{c}_{\leq k}}.\label{eq:phi}
\end{align}
Hence $\Psi_{\bm{i}}$ also extends to an isomorphism from $\mathcal{A}_{\bm{i}}$ to $\mathcal{Y}_{\bm{i}}$, denoted again by $\Psi_{\bm{i}}$. Therefore $\mathcal{L}_{\bm{i}}$ is isomorphic to $\mathcal{Y}_{\bm{i}}$ via $M_{\bm{i}}:=\Psi_{\bm{i}}\circ \Phi_{\bm{i}}^{-1}$. It remains to show that $M_{\bm{i}}(t_k)=y^{\bm{a}_{\leq k}}$ for $k=1,\dots, \ell$. By Corollary \ref{c:CA} and \eqref{eq:phi}, we have
\begin{align*}
M_{\bm{i}}(t_k)&=q^{N_k}M_{\bm{i}}\left((D_{w_{\leq k-1}\varpi_{i_{k}}, \varpi_{i_{k}}}^{\prime\; (\bm{i})})^{-1}(D_{w_{\leq k}\varpi_{i_{k}}, \varpi_{i_{k}}}^{\prime\; (\bm{i})})^{-1}\dprod_{j\in I\setminus\{i_k\}}(D_{w_{\leq k}\varpi_{j}, \varpi_{j}}^{\prime\; (\bm{i})})^{-a_{j, i_k}}\right)\\
&=q^{N_k}(y^{\bm{c}_{\leq k^-}})^{-1}(y^{\bm{c}_{\leq k}})^{-1}\dprod_{j\in I\setminus\{i_k\}}(y^{\bm{c}_{\leq k^-(j)}})^{-a_{j, i_k}}=q^{n_k}y^{\bm{a}_{\leq k}}
\end{align*}
for some $n_k\in \mathbb{Z}$, here $k^-(j):=\max(\{k'\mid k'<k, i_{k'}=j\}\cup\{0\})$, $k^-:=k^-(i_k)$ and $c_{\leq 0}:=(0,\dots, 0)$. To show that $n_k=0$, we prepare the dual bar-involution $\sigma'_{\bm{i}}$ on $\mathcal{Y}_{\bm{i}}$. Define a $\mathbb{Q}$-linear automorphism $\sigma'_{\bm{i}}\colon \mathcal{Y}_{\bm{i}}\to \mathcal{Y}_{\bm{i}}$ by $f(q)y^{\bm{a}}\mapsto f(q^{-1})y^{\bm{a}}$ for $f(q)\in \mathbb{Q}(q)$ and $\bm{a}\in \mathbb{Z}^{\ell}$. Then $\sigma'_{\bm{i}}$ satisfies 
\begin{align*}
\sigma'_{\bm{i}}(xy)=q^{(\weight x, \weight y)}\sigma'_{\bm{i}}(y)\sigma'_{\bm{i}}(x)
\end{align*}
for homogeneous elements $x, y\in \mathcal{Y}_{\bm{i}}$, where $\mathcal{Y}_{\bm{i}}$ is regarded as a $Q$-graded algebra by $\weight y_j=-\beta_j$, $j=1,\dots, \ell$. This $Q$-grading on $\mathcal{Y}_{\bm{i}}$ is compatible with the $Q$-grading on $\mathbf{A}_q[N_{-}(w)]$ via $\mathbf{A}_q[N_{-}(w)]\subset \mathcal{Y}_{\bm{i}}$ by the construction of Cauchon generators (see, for example, \cite[subsection 2.4]{GeY}). 

Recall the dual bar-involution $\sigma_{\bm{i}}$ on $\mathcal{L}_{\bm{i}}$ (and $\sigma$ on $\mathbf{A}_q[N_-^w]$) in the proof of Corollary \ref{c:CA}. Then we have $\sigma'_{\bm{i}}\circ M_{\bm{i}}=M_{\bm{i}}\circ \sigma_{\bm{i}}$ because the images of $\eta_{w, q}^{-1}([D_{w_{\leq k}\varpi_{i_k}, \varpi_{i_k}}])$ ($k=1,\dots, \ell$) under $\Phi_{\bm{i}}$ and $\Psi_{\bm{i}}$ are fixed by $\sigma_{\bm{i}}$ and $\sigma'_{\bm{i}}$, respectively, 
these elements together with their inverses generate $\mathcal{L}_{\bm{i}}$ and $\mathcal{Y}_{\bm{i}}$, respectively, and $\sigma_{\bm{i}}$ and $\sigma'_{\bm{i}}$ satisfy the same formulae with respect to the multiplication (note that $M_{\bm{i}}$ is a weight preserving isomorphism). Therefore, 
\[
q^{n_k}y^{\bm{a}_{\leq k}}=M_{\bm{i}}(t_k)=M_{\bm{i}}(\sigma_{\bm{i}}(t_k))=\sigma'_{\bm{i}}(M_{\bm{i}}(t_k))=q^{-n_k}y^{\bm{a}_{\leq k}}. 
\]
Hence we obtain $n_k=0$. %\qed
\end{proof}
\begin{remark}
We recall the notation in Introduction. Since $\Phi_{\bm{i}}$ is a quantum analogue of the torus embedding $(\mathbb{C}^{\times})^{\ell}\to N_-^w$ given by 
\[
(t_1,\dots, t_{\ell})\mapsto \exp(t_1f_{i_1})\cdots\exp(t_{\ell}f_{i_{\ell}}),
\]
Theorem \ref{t:comparison} implies that $\Psi_{\bm{i}}$ is a quantum analogue of the torus embedding $(\mathbb{C}^{\times})^{\ell}\to N_-^w$ given by 
\begin{align*}
(y_1,\dots, y_{\ell})\mapsto &\exp(y_1^{-1}f_{i_1})\cdots\exp(y_k^{-1}\prod_{j<k}y_j^{-a_{i_j, i_k}}f_{i_k})\cdots \exp(y_{\ell}^{-1}\prod_{j<\ell}y_j^{-a_{i_j, i_{\ell}}}f_{i_{\ell}})\\
&=\exp(y_1^{-1}f_{i_1})y_1^{h_{i_1}}\cdots\exp(y_k^{-1}f_{i_k})y_k^{h_{i_k}}\cdots \exp(y_{\ell}^{-1}f_{i_{\ell}})y_{\ell}^{h_{i_{\ell}}}\prod_{k=1}^{\ell}(y_{k}^{-1})^{h_{i_{k}}},
\end{align*}
here $\mathbb{C}^{\times} \to H, y\mapsto y^{h_i}$ ($i\in I$) is the $i$-th simple coroot of $H$. For the second presentation, we use the formula $y^{h_i}\exp(y'f_{j})=\exp(y^{-a_{ij}}y'f_{j})y^{h_i}$. The embedding of the latter form  appears, for example, in the theory of \emph{geometric crystals} \cite[subsection 4.4]{BK}, \cite[subsection 4.3]{Nak}. Note that this remark makes sense in the setting of symmetrizable Kac-Moody groups (see \cite{Nak}).
\end{remark}
\begin{acknowledgment*}
The author is deeply grateful to Yoshiyuki Kimura for introducing him to this subject. He would like to express his sincere gratitude to his supervisor Yoshihisa Saito for unremitting support and encouragement. He wishes to thank Bernard Leclerc, Bea Schumann and Yuki Kanakubo for enlightening comments. He is also thankful to the anonymous referees whose suggestions  significantly improve this paper.  
\end{acknowledgment*}

\def\cprime{$'$}

\end{document}